\documentclass[11pt]{amsart}

\usepackage{amsmath,amssymb,amsthm,mathtools}
\usepackage[protrusion=true,expansion=false]{microtype}
\usepackage[colorlinks=true,linkcolor=blue,citecolor=blue,urlcolor=blue]{hyperref}

\theoremstyle{plain}
\newtheorem{theorem}{Theorem}[section]
\newtheorem{proposition}[theorem]{Proposition}
\newtheorem{lemma}[theorem]{Lemma}
\newtheorem{corollary}[theorem]{Corollary}

\theoremstyle{remark}
\newtheorem{remark}[theorem]{Remark}

\theoremstyle{definition}

\newcommand{\EE}{\mathbb{E}}
\newcommand{\PP}{\mathbb{P}}
\newcommand{\R}{\mathbb{R}}

\newcommand{\cE}{\mathcal{E}}

\DeclareMathOperator{\dist}{dist}
\DeclareMathOperator{\sr}{sr}
\DeclareMathOperator{\tr}{tr}
\DeclareMathOperator{\Var}{Var}

\newcommand{\bbP}{\PP}

\newcommand{\eps}{\varepsilon}

\newcommand{\one}{\mathbf{1}}

\DeclareMathOperator{\Span}{span}
\DeclareMathOperator{\rank}{rank}

\renewcommand{\le}{\leqslant}
\renewcommand{\ge}{\geqslant}

\begin{document}

\title[Intermediate Singular Values]{Intermediate Singular Values of Random Matrices\\
under Second-Moment and Anti-Concentration Assumptions}

\author{Manuel Fernandez V}
\address{Department of Mathematics, University of Southern California, Los Angeles, CA 90089.}
\email{manuelf7@usc.edu}

\author{Achintya Raya Polavarapu}
\address{School of Mathematics, Georgia Institute of Technology, Atlanta, GA 30332.}
\email{apolavarapu6@gatech.edu}

\date{}

\begin{abstract}
Let $A=(\xi_{ij})$ be an $n\times n$ random matrix with independent, not
necessarily identically distributed, real entries satisfying
\[
  \mathbb E\xi_{ij}=0,\qquad
  \mathbb E\xi_{ij}^{2}=1,\qquad
  \sup_{z\in\mathbb R}\mathbb P(|\xi_{ij}-z|<a)\le b
\]
for fixed $a>0$ and $b\in(0,1)$. We prove that, for every
$\delta\in(0,1)$, there are constants $c,C>0$, depending only on
$a,b,\delta$, such that
\[
  \mathbb P\left(
    s_{n+1-l}(A)>Ct\frac l{\sqrt n}
  \right)
  \le
  \exp\!\left(-c\min\{tl,n\}\right)
\]
for every $t\ge1$ and every $1\le l\le(1-\delta)n$. Thus, with no moment
assumption beyond variance, all but a fixed proportion of the largest singular
values satisfy the optimal upper bound of order $l/\sqrt n$ with an
exponential upper-tail estimate. Combined with the rectangular
least-singular-value bound of Dabagia--Fern\'andez~\cite{Dabagia24}, this gives
$s_{n+1-l}(A)\asymp l/\sqrt n$ with failure
probability exponentially small in $l$. The same argument gives the
rectangular scale $\sqrt{N+1}-\sqrt{n-l+1}$ for $N\times n$ matrices whenever
$N-n+l\le(1-\delta)N$.
\end{abstract}


\maketitle

\section{Introduction}

\noindent Let
\[
  s_1(A)\ge \cdots\ge s_n(A)\ge0
\]
denote the singular values of an \(n\times n\) matrix \(A\). The smallest singular value $s_n(A)$ is a measurement of how well-conditioned the matrix is. The quantity naturally appears in the study of the stability of numerical methods, the classical example being solving the linear system $Ax=b$ where $A$ is invertible. When $b$ is replaced by the perturbation $b + \Delta b$ the perturbed solution $\tilde{x} = A^{-1}(b+\Delta b)$ can be far from the true solution $x = A^{-1}b$. On the other hand this difference is controlled by. 
\[ 
\frac{\|x-\tilde{x}\|_2}{\|\Delta b\|_2} \le \|A^{-1}\|_{op} = s_n(A)^{-1},  
\]
meaning the true solution and perturbed solutions are close when $\|\Delta b\|_2$ and $s_n(A)^{-1}$ are not too large. 
More generally, for invertible \(A\),
\[
  s_{n+1-l}(A)^{-1}=s_l(A^{-1}).
\]
Thus \(s_{n+1-l}(A)\) measures how well-conditioned $A$ is when one restricts to perturbations that are perpendicular to all but the worst $l-1$ directions. Therefore upper bounds on singular values indicate how ill-conditioned the matrix is. For a random matrix with independent and standardized entries, universality and convergence to the Marchenko-Pastur law tell us that the the natural scale for bulk singular values is of order $n^{1/2}$. In addition, for an $n \times (n-l)$ random matrix it is known that the smallest singular value is typically of order at least $l/\sqrt{n}$ so Cauchy interlacing provides the same scale for the $l$'th smallest singular value.

\medskip\noindent
The non-asymptotic study of small singular values begins with the Gaussian
work of Edelman~\cite{Edelman88,Edelman91} and Szarek~\cite{Szarek91}, and
continues through the Littlewood--Offord approach of Rudelson and
Vershynin~\cite{RV08,RV08_least,RV09}; see also
\cite{LPRT05,AGLPT08,GLPTJ17,TV09,LR12,BR17,RT18,GL21,LivTikVer,Dabagia24,
fernandezDistance2025}. For the present paper, the closest predecessor is
Wei's theorem~\cite{Feng_Wei_intermediate}, which proves the corresponding
upper bound for square matrices with i.i.d.\ standardized subgaussian entries:
\begin{equation}\label{eq:isv-upper-tail}
  \mathbb{P}\!\left(
    s_{n+1-l}(A) > C\,t\,\frac{l}{\sqrt{n}}
  \right) \le e^{-ctl}.
\end{equation}
For Gaussian matrices, sharper estimates were obtained by
Szarek~\cite{Szarek90}; for least singular values and related local questions,
see for example~\cite{RV08_least,NguyenVu18,Tatarko18,TV10_smallest,
CacciapuotiMaltsevSchlein13}.

\medskip\noindent
In this work we show that \eqref{eq:isv-upper-tail} can be recovered for matrices under much weaker assumptions than those required by Wei's theorem. In particular the assumption that the entries of the matrix have sub-gaussian tails and are identically distributed can be dropped. We only assume that the entries are standardized and satisfy a fixed-scale anti-concentration. The entries may have heavy tails and need not be identically distributed.
The price of this generality is that we must exclude an arbitrarily small fixed fraction of the singular values nearest the upper edge. 
However this is to be expected since the $l$-th smallest singular values of heavy-tailed matrices do not necessarily behave like $l/\sqrt{n}$ as $l/n \to 1$ (see \cite[Remark~1.2]{RebrovaVershynin2018} and
\cite[Section~1.1]{JainSahSawhney2022}). 
\subsection*{Main result}

\medskip\noindent
We use the following assumptions throughout. There are fixed constants $a>0$
and $b\in(0,1)$ such that, for every $i,j$,
\[
  \textnormal{(A1)}\qquad
  \mathbb{E}\xi_{ij}=0,\qquad \mathbb{E}\xi_{ij}^2=1,
\]
and
\[
  \textnormal{(A2)}\qquad
  \sup_{z\in\mathbb{R}}\mathbb{P}(|\xi_{ij}-z|<a)\le b.
\]

\begin{theorem}[Intermediate singular values away from the upper edge]
\label{thm:main}
Let $A=(\xi_{ij})_{i,j=1}^n$ be a real $n \times n$ random matrix with
independent entries satisfying \textnormal{(A1)--(A2)}. For every
$\delta\in(0,1)$,
there exist constants $c,C>0$, depending only on $a,b,\delta$, such that
\begin{equation}\label{eq:main}
  \mathbb{P}\!\left(
    s_{n+1-l}(A) > C t\frac{l}{\sqrt{n}}
  \right)
  \;\le\;
  \exp\!\left(-c\min\!\left\{tl,n\right\}\right)
\end{equation}
for every $t \ge 1$ and every integer $l$ satisfying
\[
  1 \le l \le (1-\delta)n.
\]
\end{theorem}

\medskip\noindent
For a fixed $\delta>0$, Theorem~\ref{thm:main} controls $s_j(A)$ for every
$j\ge\lceil\delta n\rceil$ and thus covers all but the largest
$\lceil \delta n \rceil$ singular values. When $tl\le cn$, the deviation profile is
$e^{-ctl}$, matching the form of Wei's subgaussian estimate.

\medskip\noindent
Some restriction near the upper edge is unavoidable under the present
assumptions. Indeed, at
$l=n$ the estimate would imply a bound on the operator norm of order $\sqrt n$ with exponentially small failure probability. Such a bound need not hold for
variance-one heavy-tailed entries, since a single exceptionally large entry
already forces the operator norm to be large. Thus $\delta$ is fixed
independently of $n$, and the constants in the theorem are allowed to depend
on $\delta$. This is also consistent with the examples in
\cite[Remark~1.2]{RebrovaVershynin2018} and
\cite[Section~1.1]{JainSahSawhney2022}. After perturbing these examples by an independent standard gaussian matrices, the resulting random matrices satisfy \textnormal{A1} and \textnormal{A2} and have, with exponentially small failure probability,  
operator norms of order at least $\sqrt{n/\delta}$, even after modifying them on a submatrix of size $\delta n \times \delta n$ .

\subsection*{Consequences}

\medskip\noindent
Together with the rectangular lower-tail estimate of Dabagia--Fern\'andez,
Theorem~\ref{thm:main} gives the following profile.

\begin{corollary}[Square intermediate singular-value profile]\label{cor:square-profile}
Under the assumptions of Theorem~\ref{thm:main}, for every $\delta\in(0,1)$
there exist constants $c,C>0$, depending only on $a,b,\delta$, such that
\[
  \mathbb P\!\left(
    c\frac l{\sqrt n}
    \le
    s_{n+1-l}(A)
    \le
    C\frac l{\sqrt n}
  \right)
  \ge
  1-Ce^{-cl}
\]
for every integer $1\le l\le(1-\delta)n$.
\end{corollary}

\begin{corollary}[Expectation upper bound]\label{cor:intermediate-expectation}
Under the assumptions of Theorem~\ref{thm:main}, for every $\delta\in(0,1)$
there exists a constant $C>0$, depending only on $a,b,\delta$, such that
\[
  \mathbb E s_{n+1-l}(A)\le C\frac l{\sqrt n}
\]
for every integer
\[
  1\le l\le(1-\delta)n.
\]
\end{corollary}

\begin{corollary}[Rectangular intermediate singular-value profile]
\label{cor:rectangular-profile}
Let $B$ be an $N\times n$ random matrix, $N\ge n$, with independent entries
satisfying
\[
  \mathbb{E}B_{ij}=0,\qquad
  \mathbb{E}B_{ij}^2=1,\qquad
  \sup_{z\in\mathbb{R}}\mathbb{P}(|B_{ij}-z|<a)\le b
\]
for some fixed $a>0$ and $b\in(0,1)$. For every
$\delta\in(0,1)$, there exist constants $c,C>0$, depending only on
$a,b,\delta$, such that if
\[
  1\le N-n+l\le(1-\delta)N,
\]
then for every $t\ge1$,
\begin{equation}\label{eq:rectangular-profile}
\begin{aligned}
&\mathbb{P}\!\left(
  ce^{-t}\bigl(\sqrt{N+1}-\sqrt{n-l+1}\bigr)
  \le s_{n+1-l}(B)\right. \\
&\hspace{3.5cm}\left.
  \le Ct\bigl(\sqrt{N+1}-\sqrt{n-l+1}\bigr)
\right)  \\
&\hspace{1.5cm}\ge
1
 -\exp\!\left(-c\min\!\left\{t(N-n+l),N\right\}\right).
\end{aligned}
\end{equation}
\end{corollary}

\subsection*{Comparison with rigidity estimates}

\medskip\noindent
Set $X=N^{-1/2}B$, and let
\[
  0\le \lambda_1\le\cdots\le\lambda_n
\]
denote the eigenvalues of $X^*X$. If $\gamma_j$ denotes the corresponding
Marchenko--Pastur classical location, then rigidity results for more regular
sample covariance matrices locate $\lambda_j$ close to $\gamma_j$ on a
microscopic scale (i.e. differences of order $n^{-1}$). For matrices with independent entries and uniformly
bounded normalized moments of every order, Bloemendal, Erd\H{o}s, Knowles,
Yau, and Yin prove Marchenko--Pastur rigidity away from the square lower edge
\cite[Theorem~2.10]{BEKYY14}. In the square case, singular-value rigidity
can be obtained by linearization under stronger moment and variance-profile
assumptions
\cite[Eq.~(1.12) and the discussion following Corollary~1.3]{AEK14}. See
also \cite[Theorem~3]{KM23} for an individual square lower-edge estimate for
an i.i.d. complex ensemble under a fourth-moment bound and a truncation
hypothesis.

\medskip\noindent
Since
\[
  s_{n+1-l}(B)=\sqrt{N\lambda_l},
\]
these results give finer information than a constant-factor estimate when
their hypotheses apply. For example, in the square case the
Marchenko--Pastur locations satisfy
\[
  \gamma_l\sim \left(\frac{\pi l}{2n}\right)^2
\]
near zero \cite[Eq.~(1.9)]{KM23}; hence rigidity with error $o(\gamma_l)$
gives
\[
  s_{n+1-l}(B)
  =\left(\frac{\pi}{2}+o(1)\right)\frac{l}{\sqrt n}.
\]
Our result is different in both assumptions and probability. The cited
rigidity theorems require stronger moment, tail, or variance-profile
conditions, and are usually stated through stochastic domination or an
overwhelming-probability event
\cite[Definition~2.2 and Theorem~3.3]{PY14}
\cite[Eq.~(1.12)]{AEK14}. Corollary~\ref{cor:rectangular-profile} instead
gives a constant-factor scale estimate under \textnormal{(A1)--(A2)} with a large deviation profile exponential in $t$ and $N-n+l$.

\subsection*{Proof strategy}

\medskip\noindent
Our proof strategy for extending Wei's theorem to the current setting is motivated by the following argument for upper bounding bulk singular values near the top edge. To upper bound $s_{2\delta n}(A)$, consider the submatrix $A'$ of $A$ obtained by removing the largest $\delta n$ columns. The sub-matrix has
Hilbert--Schmidt norm of order $n$ with high probability. Further restricting $A'$ to the subspace spanned by all but the top $\delta n$ right singular vectors, the operator norm of $A'$ restricted to this subspace is at most the order of the average length of its columns, which is of order $\sqrt{n^2/n} = \sqrt{n}$. By Cauchy interlacing this is also an upper bound on $s_{2\delta n}(A)$. For $l \ll n$ this argument is not directly applicable as the Hilbert--Schmidt norm of the sub-matrix consisting of the shortest $\delta l$ columns can be of order $\sqrt{nl}$ instead of $l/\sqrt{n}$. However if we replace $A$ with $PA'$, where $P$ is the orthogonal projection onto a fixed $O(l)$-dimensional subspace and $PA'$ is the sub-matrix of $PA$ consisting of its smallest $l$ columns, then the same argument as above implies that with high probability there exists a subspace $H$ of dimension $O(l)$ for which the restriction of $PA'$ to $H$ has operator norm of order $\sqrt{l}$.

\medskip \noindent 
Wei's proof for the upper bound of intermediate singular values of sub-gaussian matrices is based on an inverse-matrix argument whereby upper bounding $s_l(A)$ follows from lower bounding the operator norm of $A^{-1}$ on a well-conditioned subspace of dimension $O(l)$. In Wei's proof one can take $H$ to be the span of the comparison matrix $P_l^{\perp}A_l$, where $A_l$ consists of the first $l$ columns of $l$ and $P_l^{\perp}$ is the projection operator onto the orthogonal complement of the other columns of $A$. Our proof also follows this inverse-matrix argument but constructs the subspace based on the earlier observation. Roughly speaking our subspace is the span of the comparison matrix $P_{(1+2\eps)l}^{\perp}A_{(1+2\eps)l}'$, where $P_{(1+2\eps)l}^{\perp}A_{(1+2\eps)l}'$ consists of the $(1+\eps)l$ shortest columns of $P_{(1+2\eps)l}^\perp A_{(1+2\eps)l}$, restricted to its $l$ smallest right singular vectors. Because the columns of $A$ are independent $P_{(1+2\eps)}^{\perp}$ can be treated as fixed after conditioning on all other columns of $A$, while leaving $A_{(1+2\eps)l}$ unaffected. Besides the ensuing analysis, two estimates
replace the subgaussian inputs in Wei's proof: a one-scale small-ball bound
for linear images of product measures and a Hilbert--Schmidt bound for the
comparison matrix. At small codimension this bound uses the distance theorem
of Fern\'andez and weak-\(L^p\) summation; above logarithmic codimension it
follows from one-scale small ball applied directly to column-to-subspace
distances.

\subsection*{Organization of the paper}

To the best of our knowledge, this is the first optimal-scale exponential
upper-tail estimate for intermediate singular values under only mean zero,
variance one, and uniform
anti-concentration, without any subgaussian input. Section~2 fixes notation,
proves the small-ball tools, and records the external estimates and
deterministic identities used later. Section~3 proves the lower-edge
fixed-scale estimate by the inverse-matrix method. Section~4 completes the
bulk range by column trimming, recovers the deviation parameter by
monotonicity, and proves the square and rectangular consequences.

\subsection*{Acknowledgements}

The authors would like to thank Galyna Livshyts for numerous helpful
discussions and feedback. The authors used ChatGPT Pro during the preparation
of this manuscript, including for discussions that helped improve the result
to the present range of $l$, proof checking, organization, and language editing. The
authors reviewed all outputs used in the manuscript and take responsibility
for the final text.

\section{Preliminaries}\label{sec:preliminaries}

This section fixes notation and collects the inputs used in the proof.

\subsection{Standing assumptions and notation}

Throughout, $c,C>0$ denote constants depending only on the
anti-concentration parameters $a$ and $b$, unless additional dependence is
indicated. Their values may change from line to line. Dependence on
$\delta$ or other auxiliary parameters is displayed by a subscript.

For a positive integer $m$, we write $[m] := \{1,\dots,m\}$, and
$S^{m-1} \subset \mathbb{R}^m$ denotes the Euclidean unit sphere. The vectors
$e_1,\dots,e_n$ are the standard basis of $\mathbb{R}^n$. For a subspace
$E \subset \mathbb{R}^n$, we write $P_E$ for the orthogonal projection onto
$E$, and $\dist(x,E)$ for the Euclidean distance from $x$ to $E$.

For a real matrix $M$, $\|M\|$ denotes the operator norm and $\|M\|_{HS}$ the
Hilbert--Schmidt norm. If $M$ is an $N \times n$ matrix, its singular values
are written
\[
  s_1(M) \ge s_2(M) \ge \cdots \ge s_{\min(N,n)}(M) \ge 0.
\]
The Lévy concentration function of a random variable or vector $Z$ is
\[
  \mathcal{L}(Z, t)
  \;:=\;
  \sup_{u}\mathbb{P}(\|Z-u\|_2 \le t),
  \qquad t \ge 0,
\]
where the supremum is over $u \in \mathbb{R}$ if $Z$ is scalar and over
$u \in \mathbb{R}^m$ if $Z$ is vector-valued. For a nonzero matrix $D$, its
stable rank is
\[
  \sr(D) := \frac{\|D\|_{HS}^2}{\|D\|^2}.
\]
We use the convention \(\sr(0)=0\).

All entries $\xi = A_{ij}$ of the random matrix $A$ are assumed to satisfy:
\begin{enumerate}
  \item[(A1)] $\mathbb{E}\xi = 0$ and $\mathbb{E}\xi^2 = 1$.
  \item[(A2)] $\sup_{z\in\mathbb R}\PP(|\xi-z|<a)\le b$ for fixed constants
    $a > 0$ and $b \in (0,1)$.
\end{enumerate}
Unless explicitly stated otherwise, no further moment conditions are imposed.
The entries need not be identically distributed; we require only that
\textnormal{(A1)--(A2)} hold with the same constants $a,b$ for every entry.
Since $\mathcal L$ was defined using closed balls, \textnormal{(A2)} implies
$\mathcal L(\xi,a/2)\le b$. We use this consequence whenever a closed-ball
L\'evy concentration bound is invoked for an entry.

\subsection{Small-ball estimates}

We begin with the small-ball estimates used for product measures. The
subgaussian concentration estimates used in
\cite[Theorems~2.4 and~2.5]{Feng_Wei_intermediate} do not apply to
\(X_S(y)=A_Sy\) under \textnormal{(A1)--(A2)}. The scalar and matrix-valued
small-ball tools below are used in their place.

\medskip\noindent
\textit{Linear-image concentration estimates.}

The first lemma extracts from \textnormal{(A2)} a uniform small-ball bound for
normalized linear combinations. It will be applied to the coordinates of
vectors of the form \(A_Sy\).

\begin{lemma}[Scalar small ball for normalized linear combinations]
\label{lem:scalar-linear-smallball}
There exist constants $\rho=\rho(a,b)>0$ and
$\kappa=\kappa(a,b)\in(0,1)$ such that the following holds. Let
$\xi_1,\dots,\xi_m$ be independent real random variables satisfying
\[
  \sup_{z\in\mathbb R}\mathbb P(|\xi_j-z|<a)\le b
  \qquad\text{for every }j.
\]
Then, for every $x=(x_1,\dots,x_m)\in S^{m-1}$,
\[
  \mathcal L\!\left(\sum_{j=1}^m x_j\xi_j,\rho\right)\le \kappa.
\]
\end{lemma}

\begin{proof}
Fix $x\in S^{m-1}$. Choose a parameter $\eta\in(0,1)$, to be fixed below, and
set $\rho:=\eta a/2$.

If $\|x\|_\infty\ge \eta$, choose $j_0$ such that $|x_{j_0}|\ge \eta$.
Conditioning on all variables except $\xi_{j_0}$, we obtain for every
$u\in\mathbb R$,
\[
  \mathbb P\!\left(
    \left|\sum_{j=1}^m x_j\xi_j-u\right|\le\rho
    \,\middle|\,
    (\xi_j)_{j\ne j_0}
  \right)
  \le
  \sup_{z\in\mathbb R}
  \mathbb P\!\left(|\xi_{j_0}-z|\le\frac{\rho}{|x_{j_0}|}\right)
  \le b,
\]
because $\rho/|x_{j_0}|\le a/2$. Hence
\[
  \mathcal L\!\left(\sum_{j=1}^m x_j\xi_j,\rho\right)\le b.
\]

Assume now that $\|x\|_\infty<\eta$. We apply the Kolmogorov--Rogozin
concentration inequality in the standard form
\[
  \mathcal L\!\left(\sum_{j=1}^m Z_j,R\right)
  \le
  \frac{CR}
	       {\left(\sum_{j=1}^m r_j^2
	         \bigl(1-\mathcal L(Z_j,r_j)\bigr)\right)^{1/2}},
\]
see \cite[Theorem~1]{Rogozin61}. This is valid for independent random
variables $Z_j$ and parameters
$R\ge \max_j r_j\ge 0$. We use it with
\[
  Z_j:=x_j\xi_j,
  \qquad
  r_j:=\frac a2|x_j|,
  \qquad
  R:=\rho=\frac{\eta a}{2}.
\]
Since $\|x\|_\infty<\eta$, indeed $R\ge \max_j r_j$. Moreover, for every
$j$ with $x_j\ne0$,
\[
  \mathcal L(Z_j,r_j)
  =
  \sup_{z\in\mathbb R}\mathbb P\!\left(|x_j\xi_j-z|\le\frac a2|x_j|\right)
  =
  \sup_{w\in\mathbb R}\mathbb P\!\left(|\xi_j-w|\le\frac a2\right)
  \le b.
\]
For indices with $x_j=0$, the corresponding term in the
Kolmogorov--Rogozin denominator is zero and may be ignored.
Therefore
\[
	  \mathcal L\!\left(\sum_{j=1}^m x_j\xi_j,\rho\right)
	  \le
	  \frac{C\eta a/2}
	       {\left(\sum_{j=1}^m (a^2/4)x_j^2(1-b)\right)^{1/2}}
	  =
	  \frac{C\eta}{\sqrt{1-b}}.
\]
Choose $\eta=\eta(a,b)$ small enough so that
\[
  \frac{C\eta}{\sqrt{1-b}}\le \frac{1+b}{2}.
\]
Then both cases are bounded by
\[
  \kappa:=\frac{1+b}{2}\in(0,1),
\]
which proves the lemma.
\end{proof}

The Bernoulli form of the linear-image estimate is used only in the
symmetrization argument below. Wei uses subgaussian linear-image small-ball
estimates; under \textnormal{(A1)--(A2)} we need the weaker fixed-scale estimate
below, and the stable-rank lower bound later makes it sufficient.

\begin{theorem}[Stable-rank small-ball estimate for Rademacher images]
\label{thm:rademacher-stable-rank-small-ball}
There exist absolute constants $\eta_{\mathrm{rad}}, c_{\mathrm{rad}} > 0$
such that the following holds. Let
$\varepsilon=(\varepsilon_1,\dots,\varepsilon_N)$ be a vector of independent
Rademacher random variables, and let $D$ be a deterministic real matrix. Then
\[
  \sup_{u}
  \mathbb P_\varepsilon\!\left(
    \|D\varepsilon-u\|_2
    \le
    \eta_{\mathrm{rad}}\|D\|_{HS}
  \right)
  \le
  2\exp\!\left(-c_{\mathrm{rad}} \sr(D)\right).
\]
\end{theorem}

\begin{proof}
The case $D=0$ is trivial. Assume $D\ne0$.
Put $h=\|D\|_{HS}$, $\sigma=\|D\|$, and $r=h^2/\sigma^2$. Fix $u$ and set
\[
  f(x):=\|Dx-u\|_2,\qquad x\in\mathbb R^N.
\]
The function $f$ is convex and $\sigma$-Lipschitz. Let $m$ be a median of
$f(\varepsilon)$. Talagrand's convex concentration inequality on the discrete
cube~\cite[Theorem~6.6]{Talagrand96NewLook} gives
\[
  \mathbb P\bigl(f(\varepsilon)\ge m+s\bigr)
  \le 2\exp\!\left(-c\frac{s^2}{\sigma^2}\right),
  \qquad
  \mathbb P\bigl(f(\varepsilon)\le m-s\bigr)
  \le 2\exp\!\left(-c\frac{s^2}{\sigma^2}\right)
\]
for all $s\ge0$. Hence
\[
  \mathbb E(f(\varepsilon)-m)_+^2\le C\sigma^2.
\]
Since
\[
  \mathbb E f(\varepsilon)^2
  =
  \mathbb E\|D\varepsilon-u\|_2^2
  =
  h^2+\|u\|_2^2
  \ge h^2,
\]
and $f\le m+(f-m)_+$, we get
\[
  \mathbb E f^2\le 2m^2+2\mathbb E(f-m)_+^2.
\]
Thus
\[
  h^2\le 2m^2+C\sigma^2.
\]
If $r\ge r_0$ for a large enough absolute constant $r_0$, then $m\ge h/2$.
Taking $\eta_{\mathrm{rad}}\le1/4$ and using the lower-tail estimate,
\[
  \mathbb P\bigl(f(\varepsilon)\le \eta_{\mathrm{rad}}h\bigr)
  \le
  2\exp(-c r).
\]
If $1\le r<r_0$, the same bound is trivial after decreasing $c$. Taking the
supremum over $u$ proves the theorem.
\end{proof}

Combining the scalar estimate with symmetrization gives the form used later in
the proof. Compared with the usual anisotropic concentration bound, the scale
is allowed to shrink by a constant depending on the one-dimensional
anti-concentration parameters; this makes the estimate effective for any
fixed \(\kappa<1\).

\begin{proposition}[Stable-rank small ball for linear images of product measures]
\label{prop:one-scale-anisotropic}
Let $Z=(Z_1,\dots,Z_N)$ have independent real coordinates satisfying
\[
  \mathbb E Z_i^2 \le 1
  \qquad\text{and}\qquad
  \mathcal L(Z_i,\rho)\le \kappa<1
  \quad\text{for all }i.
\]
Then there exist constants $\eta=\eta(\rho,\kappa)>0$ and
$c=c(\rho,\kappa)>0$ such that, for every deterministic real matrix $D$,
\[
  \mathcal L(DZ,\eta\|D\|_{HS})
  \le
  2\exp\!\left(-c\,\sr(D)\right).
\]
\end{proposition}

\begin{proof}
If $D=0$, the claim is trivial. Assume $D\ne0$.
Let $Z'$ be an independent copy of $Z$, and set
\[
  \widetilde Z:=Z-Z'.
\]
For every $s>0$ and every center $u$,
\[
  \mathbb P(\|DZ-u\|_2\le s)^2
  \le
  \mathbb P(\|D\widetilde Z\|_2\le 2s).
\]
Thus it is enough to control the small-ball probability of $D\widetilde Z$.

For each coordinate, conditional on $Z_i'$, the assumption
$\mathcal L(Z_i,\rho)\le \kappa$ gives
\[
  \mathbb P(|Z_i-Z_i'|\ge \rho\mid Z_i')\ge 1-\kappa.
\]
Also,
\[
  \mathbb E|Z_i-Z_i'|^2 = 2\Var(Z_i)\le 2\mathbb E Z_i^2 \le 2.
\]
Choose
\[
  M:=\frac{2}{\sqrt{1-\kappa}},
  \qquad
  \delta:=\frac{1-\kappa}{2}.
\]
Then
\[
  \mathbb P(\rho\le |\widetilde Z_i|\le M)
  \ge
  1-\kappa-\frac{2}{M^2}
  =
  \delta.
\]
Let
\[
  \Gamma_i:=\mathbf 1_{\{\rho\le |\widetilde Z_i|\le M\}},
  \qquad
  G:=\{i:\Gamma_i=1\}.
\]
For a fixed matrix $D$, put
\[
  w_i:=\|De_i\|_2^2,
  \qquad
  W:=\sum_i w_i=\|D\|_{HS}^2.
\]
Since $0\le \Gamma_i w_i\le w_i\le \|D\|^2$ and
$\mathbb E(\Gamma_i w_i)\ge \delta w_i$, we have
\[
  \sum_i w_i^2\le \left(\max_i w_i\right)\sum_iw_i
  \le \|D\|^2\|D\|_{HS}^2,
\]
and Hoeffding's inequality yields
\[
  \mathbb P\!\left(
    \sum_{i\in G} w_i < \frac{\delta}{2}W
  \right)
  \le
  \exp\!\left(
    -c\delta^2\frac{W}{\|D\|^2}
  \right)
  =
  \exp\!\left(-c\delta^2 \sr(D)\right).
\]
Call the complementary event $\mathcal G$.

Condition on the set $G$, on the magnitudes $|\widetilde Z_i|$ for $i\in G$,
and on all coordinates $\widetilde Z_i$ with $i\notin G$. Since
$\widetilde Z_i$ is symmetric, the signs of $\widetilde Z_i$ on $G$ are
independent Rademacher variables. Therefore, conditionally,
\[
  D\widetilde Z
  =
  v
  +
  F\varepsilon,
\]
where $v$ is deterministic under the conditioning,
$\varepsilon=(\varepsilon_i)_{i\in G}$ is a Rademacher vector, and
\[
  F
  =
  D_G\operatorname{diag}(|\widetilde Z_i|:i\in G).
\]
On $\mathcal G$,
\[
  \|F\|_{HS}^2
  =
  \sum_{i\in G} |\widetilde Z_i|^2\|De_i\|_2^2
  \ge
  \rho^2\sum_{i\in G}w_i
  \ge
  \frac{\rho^2\delta}{2}\|D\|_{HS}^2,
\]
while
\[
  \|F\|\le M\|D\|.
\]
Hence
\[
  \sr(F)
  =
  \frac{\|F\|_{HS}^2}{\|F\|^2}
  \ge
  \frac{\rho^2\delta}{2M^2}\,\sr(D).
\]

Choose
\[
  \eta
  :=
  \frac{\eta_{\mathrm{rad}}\rho}{4}\sqrt{\frac{\delta}{2}}.
\]
Then on $\mathcal G$,
\[
  2\eta\|D\|_{HS}
  \le
  \eta_{\mathrm{rad}}\|F\|_{HS}.
\]
By Theorem~\ref{thm:rademacher-stable-rank-small-ball},
\[
  \mathbb P_\varepsilon\!\left(
    \|F\varepsilon+v\|_2
    \le
    2\eta\|D\|_{HS}
  \right)
  \le
  2\exp\!\left(-c_{\mathrm{rad}}\sr(F)\right)
  \le
  2\exp\!\left(-c \sr(D)\right),
\]
where $c=c(\rho,\kappa)>0$.

Combining this conditional estimate with the bound on $\mathcal G^c$ gives
\[
  \mathbb P(\|D\widetilde Z\|_2\le 2\eta\|D\|_{HS})
  \le
  3\exp(-c \sr(D)).
\]
Taking square roots in the symmetrization inequality and decreasing $c$ if
necessary yields
\[
  \mathcal L(DZ,\eta\|D\|_{HS})
  \le
  2\exp(-c \sr(D)).
\]
\end{proof}

\medskip\noindent
The estimate gives the exponent \(\sr(D)\). Thus, once a comparison matrix
has stable rank of order \(l\), a fixed-scale small-ball event for its image
has probability at most \(e^{-cl}\). No tail estimate for the coordinates of
\(Z\) is required.

\subsection{Distance and rectangular invertibility estimates}

The next two inputs are lower-tail estimates. The distance theorem is used only
when the relevant codimension is small. The rectangular estimate is used both
inside the inverse-matrix argument and later to obtain the matching lower
bound.

The following is \cite[Theorem~1.1]{fernandezDistance2025}.
\begin{theorem}[Distance to a random subspace]
\label{distance_theorem}
There exist constants $c, C, \lambda > 0$ depending only on $a$ and $b$ such
that the following holds for all $1 \le d \le \lambda n/\log n$. Let
$X \in \mathbb{R}^n$ be a random vector with independent entries satisfying
\textnormal{(A1)--(A2)}, and let
$M \in \mathbb{R}^{n \times (n-d)}$ be an independent random matrix whose
columns satisfy \textnormal{(A1)--(A2)}. Let $H$ denote the column span of
$M$. Then, for every $t\ge0$,
\[
  \mathcal{L}(P_{H^\perp}X,\, t\sqrt{d})
  \;\le\; (Ct)^d + e^{-cn}.
\]
\end{theorem}

\begin{remark}
The codimension restriction $d \le \lambda n/\log n$ in
Theorem~\ref{distance_theorem} is used only in the small-codimension branch of
the proof. In the larger range we use the one-scale anisotropic small-ball
estimate instead.
\end{remark}

The following is \cite[Theorem~1]{Dabagia24}.
\begin{theorem}[Rectangular least singular value]\label{smallest_sing_val_lower}
Let $a > 0$ and $b \in (0,1)$. There exist $C, c > 0$, depending only on
$a$ and $b$, such that the following holds. Let
$M \in \mathbb{R}^{N \times n}$ be a random matrix with independent entries
$\xi$ satisfying
\[
  \mathbb{E}\xi = 0,\qquad
  \mathbb{E}\xi^2 = 1,\qquad
  \sup_{z}\mathbb{P}(|\xi - z| < a) \le b.
\]
Then there is $\varepsilon_0=\varepsilon_0(a,b)>0$ such that, for every
$0<\varepsilon\le\varepsilon_0$,
\[
  \mathbb{P}\!\left(
    s_n(M) \le \varepsilon\!\left(\sqrt{N+1} - \sqrt{n}\right)
  \right)
  \;\le\;
  \bigl(C\varepsilon\log(1/\varepsilon)\bigr)^{N-n+1} + e^{-cN}.
\]
\end{theorem}

\begin{remark}
Theorem~\ref{smallest_sing_val_lower} is used in the displayed form above; no
additional \(2+\beta\) moment assumption enters through this input. The \(\log(1/\varepsilon)\) factor imposes only the internal constraint \(\alpha > C\log\alpha\) on one proof parameter, which can be met together with the other choices of constants.
\end{remark}

\subsection{Hilbert--Schmidt nets}

The net used below is chosen for Hilbert--Schmidt approximation rather than operator-norm approximation. This avoids requiring an operator-norm bound for the reduced map under only two moments.

The following is the form of \cite[Theorem~3]{GL21} used below.
\begin{theorem}[Net with Hilbert--Schmidt approximation]
\label{thm:hs-net}
There exists an absolute constant $C_{\mathrm{net}}>0$ such that for every integer $m\ge1$ and every $\varepsilon\in(0,1)$ there exists a deterministic set
\[
  \mathcal N_\varepsilon
  \subset
  \frac32 B_2^m\setminus \frac12 B_2^m
\]
with
\[
  |\mathcal N_\varepsilon|
  \le
  \left(\frac{C_{\mathrm{net}}}{\varepsilon}\right)^m
\]
and the following approximation property: for every matrix $M$ with $m$ columns and every $y\in S^{m-1}$ there exists $y_i\in\mathcal N_\varepsilon$ such that
\begin{equation}\label{eq:lattice-approx}
  \|M(y-y_i)\|_2
  \le
  \frac{\sqrt2\,\varepsilon}{\sqrt m}\|M\|_{HS}.
\end{equation}
\end{theorem}

\subsection{Biorthogonal identities}

We also need the deterministic identities connecting these estimates to the inverse matrix.

The deterministic algebra used below is from Feng Wei's biorthogonal construction~\cite[Proposition~2.1, Lemma~2.2, and
Corollary~2.3]{Feng_Wei_intermediate}. We state the basic proposition here. Its two block-structured consequences needed later will be invoked after the proof-specific notation has been fixed.

The following is \cite[Proposition~2.1]{Feng_Wei_intermediate}.
\begin{proposition}[Biorthogonal systems]
\label{prop:biorthogonal}
\hfill
\begin{enumerate}
  \item[\textnormal{(1)}] Let $D$ be an invertible $n \times n$ matrix with columns $v_k = De_k$. Define $v_k^* = (D^{-1})^* e_k$. Then
  $(v_k, v_k^*)_{k=1}^n$ is a complete biorthogonal system in
  $\mathbb{R}^n$.
  \item[\textnormal{(2)}] Every linearly independent system $(v_k)_{k=1}^n$
  in an $n$-dimensional Hilbert space $H$ admits a unique dual system
  $(v_k^*)_{k=1}^n$ such that $(v_k, v_k^*)$ is a complete biorthogonal
  system.
  \item[\textnormal{(3)}] In a complete biorthogonal system, $\|v_k^*\|_2 = 1/\dist(v_k, \Span(v_j:j \neq k))$.
\end{enumerate}
\end{proposition}

The next lemma records the block form of
\cite[Lemma~2.2 and Corollary~2.3]{Feng_Wei_intermediate}.
\begin{lemma}[Block biorthogonal identities]
\label{lem:block-biorthogonal}
Let $D$ be an invertible $n\times n$ matrix with columns $v_1,\dots,v_n$, and
let $[n]=S\sqcup T$ be a partition with $|S|=l$. Write
\[
  V_S := (v_j)_{j\in S},
  \qquad
  H_T := \Span(v_j:j\in T),
  \qquad
  P_T^\perp := P_{H_T^\perp}.
\]
Let $v_k^*=(D^{-1})^*e_k$ and $y_k^*:=P_{H_T}v_k^*$ for $k\in T$, and let
$B_T$ be the $(n-l)\times n$ matrix whose rows are $(y_k^*)^T$, indexed by
$k\in T$. Then the following hold.
\begin{enumerate}
  \item[\textnormal{(1)}] The matrix $B_T$ is determined entirely by the columns $\{v_j:j\in T\}$.
  \item[\textnormal{(2)}] For every $y\in\mathbb R^l$,
  \[
    \|D^{-1}P_T^\perp V_S y\|_2^2
    =
    \|y\|_2^2+\|B_TV_S y\|_2^2.
  \]
  In particular,
  \[
    \|D^{-1}P_T^\perp V_S y\|_2 \ge \|B_TV_S y\|_2.
  \]
  \item[\textnormal{(3)}] For every subcollection $V_J$ of columns of $V_S$,
  \[
    \|D^{-1}P_T^\perp V_J\|_{HS}^2
    =
    |J|+\|B_TV_J\|_{HS}^2.
  \]
  \item[\textnormal{(4)}] If $V_T:=(v_j)_{j\in T}$, then
  $B_TV_T=I$, $B_T$ vanishes on $H_T^\perp$, and
  \[
    \|B_T\|=s_{\min}(V_T)^{-1}.
  \]
\end{enumerate}
\end{lemma}

\medskip\noindent
For \textnormal{(4)}, the identity $B_TV_T=I$ follows from biorthogonality,
and the rows of $B_T$ lie in $H_T$; hence $B_T$ is the Moore--Penrose left
inverse of $V_T$.

\begin{lemma}[Singular-subspace reduction]\label{lem:singular-subspace-reduction}
Let $D$ be an invertible $n\times n$ matrix, and let
$U:\mathbb R^m\to\mathbb R^n$ be a linear map. Assume that
\[
  \inf_{y\in S^{m-1}}\|D^{-1}Uy\|_2\ge a
  \qquad\text{and}\qquad
  \|U\|_{HS}\le H.
\]
Then there is a subspace $E\subseteq \operatorname{range}(U)$ with
$\dim(E)\ge m/2$ such that
\[
  \inf_{u\in S_E}\|D^{-1}u\|_2
  \ge
  \frac{a\sqrt m}{\sqrt2\,H}.
\]
\end{lemma}

\begin{proof}
The first assumption implies that $U$ is injective. Let
\[
  \sigma_1(U)\ge\cdots\ge\sigma_m(U)>0
\]
be the singular values of $U$. Since $\sum_j\sigma_j(U)^2=\|U\|_{HS}^2$, at
least $\lceil m/2\rceil$ singular values are at most $\sqrt2H/\sqrt m$. Let
$F\subseteq\mathbb R^m$ be the span of the corresponding right singular
vectors and set $E:=U(F)$. Then $\dim(E)=\dim(F)\ge m/2$, and
\[
  \|Uy\|_2\le \frac{\sqrt2H}{\sqrt m}
  \qquad\text{for every }y\in S_F.
\]
Every $u\in S_E$ has the form $u=Uy/\|Uy\|_2$ for some $y\in S_F$, so
\[
  \|D^{-1}u\|_2
  =
  \frac{\|D^{-1}Uy\|_2}{\|Uy\|_2}
  \ge
  \frac{a\sqrt m}{\sqrt2\,H}.
\]
\end{proof}

The lower-edge estimate begins with the oversampled block and the comparison
matrices. The next section establishes the uniform stable-rank, reduced-block,
and inverse bounds needed for the min--max conclusion.

\section{The Lower-Edge Fixed-Scale Estimate}\label{sec:fixed-scale}

We first prove the fixed-scale estimate in the lower-edge range
$l\le c_0n$. This is the part of the proof that requires the inverse matrix,
the biorthogonal construction, and the Hilbert--Schmidt net. The extension
from $c_0n$ to every $l\le(1-\delta)n$ is carried out separately in
Section~\ref{sec:rescaling-consequences}.

\begin{proposition}[Lower-edge fixed-scale estimate]\label{prop:fixed-scale}
There exist constants $c_0, C_1, C_2 > 0$ depending only on $a$ and $b$ such
that for every integer $l$ satisfying
\[
  1 \le l \le c_0 n,
\]
one has
\begin{equation}\label{eq:claim1}
  \mathbb{P}\!\left(s_{n+1-l}(A) \ge C_1\frac{l}{\sqrt{n}}\right)
  \;\le\; \exp(-C_2 l).
\end{equation}
\end{proposition}

\subsection{Setup for the fixed-scale proof}\label{subsec:proof-setup}

Fix the proof parameter $l$. We prove a lower bound on
$s_{l'}(A^{-1})$, where $l' = \lfloor\lambda l\rfloor$ and
$\lambda\in(0,1)$ is a small constant depending only on $a$ and $b$. This is
equivalent to an upper bound on $s_{n+1-l'}(A)$. We first select a block of
$l$ columns from a slightly larger block. From it we later extract
$l'$ columns and study the ratio
\[
  \frac{\|A^{-1}P_l^\perp A_{l'} y\|_2}{\|P_l^\perp A_{l'} y\|_2},
\]
not on the whole reduced sphere, but on a large singular subspace of
$P_l^\perp A_{l'}$. The notation for the selected sets is introduced at the
point where it is used.

The argument below is used only for proof parameters $l\ge 2/\lambda$. This
ensures $l'\ge1$ and, since we take $\lambda\le1/2$, also $l/2\ge2$ in the
weak-$L^{l/2}$ estimate. The final monotonicity step chooses such a proof
parameter even when the target index is small.

It suffices to prove the result when $A$ is almost surely invertible. The
general case follows by regularization: let $G$ be an independent standard Gaussian
matrix and set
\[
  A_\tau:=\frac{A+\tau G}{\sqrt{1+\tau^2}},
  \qquad 0<\tau\le1.
\]
Then $A_\tau$ has centered variance-one entries and
$\mathcal L((A_\tau)_{ij},a/(2\sqrt2))\le b$; this follows by conditioning on
$G$, since the corresponding interval for $A_{ij}$ has radius at most $a/2$.
The constants below are uniform in $\tau$. If the estimate is proved for
$A_\tau$ with threshold $C l/\sqrt n$, then for every $\eta>0$,
singular-value continuity and Fatou's lemma give the same estimate for $A$
with threshold $(C+\eta)l/\sqrt n$. Increasing $C$ removes the extra
$\eta$. For the rest of this section, $l$ is fixed. We choose an
oversampling constant
$\vartheta=\vartheta(a,b)>0$ in the small-ball estimate below, set
\[
  m_l:=\lceil(1+\vartheta)l\rceil,
\]
and denote by $c_\vartheta>0$ any constant such that
\[
  \binom{m_l}{l}\le e^{c_\vartheta l}
  \qquad\text{for all }l\ge1.
\]
By the standard binomial estimate, we can make $c_\vartheta\to0$ as
$\vartheta\downarrow0$. The constants are fixed in the following order. First
$\alpha$ is chosen in the operator-norm bound for the comparison matrix. Then
$\vartheta$ is chosen small enough to absorb the binomial counting terms. Next
$K$ is fixed in the reduced Hilbert--Schmidt selection, the net radius
$\varepsilon$ is chosen after the resulting Hilbert--Schmidt constant is known,
$\lambda$ is chosen small enough for the reduced-block counting terms, the floor
estimates, and the absorption of fixed exponential prefactors, and finally
$c_0$ is chosen. This order is used below without further comment.

\subsection{Pointwise bounds for the selected block}\label{subsec:pointwise-selected}

For a fixed unit vector $y\in S^{l-1}$ we need two estimates: an
upper bound on $\|P_l^\perp A_l y\|_2$ and a lower bound on
$\|A^{-1}P_l^\perp A_l y\|_2$. These are the pointwise inputs that will later
be combined with Hilbert--Schmidt control of the reduced block.

\subsubsection{Selection notation}

Let $1\le l\le c_0 n$ be fixed, where $c_0>0$ will be chosen small enough at
the end.
Write $A_{m_l}:=(A^1,\dots,A^{m_l})$ for the first $m_l$ columns of $A$. Define
\[
  H_{m_l}:=\Span(A^{m_l+1},\dots,A^n),
  \qquad
  P_{m_l}^\perp:=\text{orthogonal projection onto }H_{m_l}^\perp.
\]
The projection $P_{m_l}^\perp$ is determined entirely by the columns
$\{A^{m_l+1},\dots,A^n\}$ and is therefore independent of $A_{m_l}$. Moreover,
\[
  \mathrm{rank}(P_{m_l}^\perp)=m_l.
\]

\medskip\noindent
Let $S\subseteq[m_l]$ index the $l$ smallest values of
$\|P_{m_l}^\perp A^i\|_2$, breaking ties by increasing index. Define
$A_l:=(A^i)_{i\in S}$ and set $T:=[n]\setminus S$, so $|T|=n-l$.
Because $P_{m_l}^\perp$ is independent of $A_{m_l}$, the selection depends on
the matrix only through the projected norms of the first $m_l$ columns.

\medskip\noindent
Given the selection $S$, define
\[
  H_l:=\Span(A^j:\,j\in T),
  \qquad
  P_l^\perp:=\text{orthogonal projection onto }H_l^\perp.
\]
Since $T\supseteq\{m_l+1,\dots,n\}$, we have $H_{m_l}\subseteq H_l$, and therefore
\begin{equation}\label{eq:proj-inclusion}
  H_l^\perp \;\subseteq\; H_{m_l}^\perp,
  \qquad
  \|P_l^\perp v\|_2 \;\le\; \|P_{m_l}^\perp v\|_2
  \quad \text{for all }v\in\R^n.
\end{equation}

\medskip\noindent
Define $X_k^*:=(A^{-1})^*e_k$ and $Y_k^*:=P_l X_k^*$ for each $k\in T$, and
let $B$ be the $(n-l)\times n$ matrix whose $k$-th row is $(Y_k^*)^T$,
indexed over $k\in T$. Applying Lemma~\ref{lem:block-biorthogonal} to the
partition $(S,T)$, and setting $U(y):=P_l^\perp A_l y$, we obtain for every
$y\in\mathbb R^l$,
\begin{equation}\label{eq:block-biorth-step1}
  \|A^{-1}U(y)\|_2^2
  =
  \|y\|_2^2+\|B(A_ly)\|_2^2
  \ge
  \|B(A_ly)\|_2^2.
\end{equation}
For deterministic $S$, the corresponding matrix $B_S$ is determined solely by
$\{A^j:j\in T\}$ and is independent of $A_S$. Since the present $S$ is selected
from the matrix, the independence is used only through estimates proved
uniformly over deterministic choices of $S$.

\subsubsection{Selecting a short projected block}

Choose $l$ columns among the first $m_l$ whose projection onto
$H_{m_l}^\perp$ is short in Hilbert--Schmidt norm. Since operator norm is
bounded by Hilbert--Schmidt norm, this also controls $\|U(y)\|_2$ for every
unit vector $y$.

\begin{proposition}\label{prop:short-columns}
There exist constants $t_0=t_0(\vartheta)>0$ and $c_1=c_1(\vartheta)>0$ such
that, with
\[
  I := \{i:\|P_{m_l}^\perp A^i\|_2^2 \le t_0l\} \subseteq [m_l],
\]
one has
\begin{equation}\label{eq:short-column}
  \bbP\left(|I| \ge l\right) \ge 1 - e^{-c_1 l}.
\end{equation}
\end{proposition}
\begin{proof}
Write $\beta:=1+\vartheta$. For each $i\in[m_l]$, conditioning on
$\{A^{m_l+1},\dots,A^n\}$ makes $P_{m_l}^\perp$ deterministic. By isotropy and
the fact that $\mathrm{rank}(P_{m_l}^\perp)=m_l\le 2\beta l$,
\[
  \EE\bigl[\|P_{m_l}^\perp A^i\|_2^2\bigr]
  \;=\;\mathrm{tr}(P_{m_l}^\perp)
  \;=\;m_l
  \;\le\; 2\beta l.
\]
\[
  \bbP \left(\|P_{m_l}^\perp A^i\|_2^2 \ge tl\right) \le \frac{2\beta}{t}
  \qquad\text{for every }t\ge1.
\]
Choose $t_0=t_0(\vartheta)$ so large that
\[
  p_\vartheta:=\frac{2\beta}{t_0}<\frac{\vartheta}{16\beta}.
\]
Then, conditionally on $\{A^j:j>m_l\}$, the number of bad columns among the
first $m_l$ is stochastically dominated by $\mathrm{Bin}(m_l,p_\vartheta)$.
Since
\[
  m_l-l\ge \vartheta l,
  \qquad
  m_l p_\vartheta \le 2\beta l\cdot \frac{\vartheta}{16\beta}
  = \frac{\vartheta l}{8},
\]
a Chernoff bound yields
\[
  \bbP\!\left(|I|<l \,\middle|\, \{A^j:j>m_l\}\right)
  =
  \bbP\!\left(m_l-|I|>m_l-l \,\middle|\, \{A^j:j>m_l\}\right)
  \le e^{-c_1(\vartheta)l}.
\]
Averaging over the outside columns proves \eqref{eq:short-column}.
\end{proof}

\begin{remark}
The point here is combinatorial. A single projected column is controlled only
with constant probability under second moments, but the oversampled block
contains slightly more than $l$ independent columns and we only need $l$ of
them to be good. This gives exponential decay in $l$ without any subgaussian
input.
\end{remark}

\medskip\noindent
On the event from Proposition~\ref{prop:short-columns}, every selected column
of $A_l$ satisfies $\|P_{m_l}^\perp A^i\|_2^2 \le t_0 l$, and therefore
\[
  \|P_{m_l}^\perp A_l\|_{HS}^2 \le t_0 l^2.
\]
Using \eqref{eq:proj-inclusion} and the elementary bound
$\|M\| \le \|M\|_{HS}$ for every matrix $M$, we obtain
\[
  \|P_l^\perp A_l\|
  \;\le\;
  \|P_l^\perp A_l\|_{HS}
  \;\le\;
  \|P_{m_l}^\perp A_l\|_{HS}
  \;\le\;
  \sqrt{t_0}\,l.
\]
It follows that there exist constants $C,c_1>0$, depending only on $a$ and $b$,
such that
\begin{equation}\label{eq:step11-conclusion}
  \bbP\!\left(
    \|P_l^\perp A_l\| > Cl
    \ \text{ or }\
    \|P_l^\perp A_l\|_{HS} > Cl
  \right)
  \;\le\;
  e^{-c_1 l}.
\end{equation}

Define the event
\begin{equation}\label{eq:E1}
  \cE_1 := \bigl\{A:
    \|P_l^\perp A_l\| > Cl
    \text{ or }
    \|P_l^\perp A_l\|_{HS} > Cl
  \bigr\}.
\end{equation}
By \eqref{eq:step11-conclusion}, $\bbP(\cE_1)\le e^{-c_1 l}$. In
particular, on $\cE_1^c$ one has
\[
  \|P_l^\perp A_l\|_{HS} \le Cl
  \qquad\text{and}\qquad
  \|U(y)\|_2 = \|P_l^\perp A_l y\|_2 \le Cl
  \qquad\text{for every } y \in S^{l-1}.
\]

\subsubsection{Lower bound on the inverse image}

We need uniform control of $\|B\|_{HS}$ and $\|B\|$ over all possible choices
of $A_l\subseteq A_{m_l}$. The Hilbert--Schmidt and operator-norm estimates
are first proved for a fixed deterministic set
$S\subseteq[m_l]$, $|S|=l$, and then union-bounded over the
$\binom{m_l}{l}\le e^{c_\vartheta l}$ possibilities. The final part gives the
lower bound on $\|B A_l y\|_2$.

\medskip\noindent
\textit{Lower Hilbert--Schmidt bound for \(B\).}

We use the standard identity \cite[Step~1.2.1]{Feng_Wei_intermediate}
\begin{equation}\label{eq:BHS-identity}
  \|B\|_{HS}^2
  \;=\;\sum_{k\in T}\|Y_k^*\|_2^2
  \;=\;\sum_{k\in T}\dist(A^k,H_{l,k})^{-2},
\end{equation}
where $H_{l,k}:=\Span(A^j:\,j\in T,\,j\ne k)$.

\medskip\noindent
For any index set $S\subseteq[m_l]$ and any $k\ge m_l+1$, we have
$k\in T=[n]\setminus S$ and
\[
  H_{l,k}
  =\Span(A^j:\,j\in T,\,j\ne k)
  \;\supseteq\;
  \Span(A^j:\,j>m_l,\,j\ne k)
  =:H_{m_l,k},
\]
since $T\supseteq\{m_l+1,\dots,n\}$. Distance to a larger subspace is smaller, so
\[
  \dist(A^k,H_{l,k})^{-2}
  \;\ge\;
  \dist(A^k,H_{m_l,k})^{-2}
  \qquad \text{for all }k\ge m_l+1.
\]
Summing over $k\ge m_l+1$ and using \eqref{eq:BHS-identity},
\begin{equation}\label{eq:BHS-lower-indep}
  \|B\|_{HS}^2
  \;\ge\;
  \sum_{k=m_l+1}^n \dist(A^k,H_{m_l,k})^{-2}.
\end{equation}
The right-hand side depends only on $\{A^k:k>m_l\}$ and is
therefore independent of the choice of $S\subseteq[m_l]$. It therefore
suffices to prove a lower bound for this quantity alone.

\medskip\noindent
Shrinking the final range constant $c_0$ if necessary, we may assume
$(2+\vartheta)c_0<1/2$. Since $l\le c_0n$, this guarantees that the outside
columns contain linearly many blocks and, in particular,
\[
  m:=\left\lfloor\frac{n-m_l}{l}\right\rfloor\ge c_\vartheta'\frac nl
\]
for a constant $c_\vartheta'>0$. Partition
$\{m_l+1,\dots,m_l+ml\}$ into $m$ disjoint blocks of length $l$:
\[
  \mathcal{I}_j
  :=\{m_l+(j-1)l+1,\dots,m_l+jl\},
  \qquad j=1,\dots,m.
\]

\medskip\noindent
For $k\in\mathcal{I}_j$, the subspace $H_{m_l,k}$ contains columns from other
blocks $\mathcal{I}_{j'}, j'\ne j$, which creates a conditional dependence
among the distances $\{\dist(A^k,H_{m_l,k}):k\in\mathcal{I}_j\}$. To
restore independence, we replace $H_{m_l,k}$ by the external subspace for
block $j$:
\[
  H_j^{\mathrm{ext}}
  :=\Span(A^r:\,r>m_l,\,r\notin\mathcal{I}_j),
  \qquad
  \widetilde{Z}_k:=\dist(A^k,H_j^{\mathrm{ext}}),
  \quad k\in\mathcal{I}_j.
\]
Since $H_j^{\mathrm{ext}}\subseteq H_{m_l,k}$, we have
$\widetilde{Z}_k\ge\dist(A^k,H_{m_l,k})$ and thus
\begin{equation}\label{eq:Ztilde-lower}
  \sum_{k\in\mathcal{I}_j}\dist(A^k,H_{m_l,k})^{-2}
  \;\ge\;
  \sum_{k\in\mathcal{I}_j}\widetilde{Z}_k^{-2}.
\end{equation}
Conditional on $H_j^{\mathrm{ext}}$ (i.e., on all columns
$\{A^r:r>m_l,r\notin\mathcal{I}_j\}$), the subspace $H_j^{\mathrm{ext}}$ is
deterministic and each $\widetilde{Z}_k$ depends only on $A^k$; since the
columns $\{A^k:k\in\mathcal{I}_j\}$ are mutually independent, so are
$\{\widetilde{Z}_k:k\in\mathcal{I}_j\}$ conditional on $H_j^{\mathrm{ext}}$.

\begin{lemma}[Each block is good with exponentially high probability]
\label{lem:block-good}
Let
\[
  \mathcal{E}_j
  := \left\{
    \#\{k\in\mathcal{I}_j:\widetilde{Z}_k\le C_2\sqrt{l}\}\ge l/2
  \right\}
\]
for a constant $C_2=C_2(\vartheta)>0$ to be chosen. Then
\[
  \bbP(\mathcal{E}_j^c)\;\le\;e^{-c_3 l}
\]
for a constant $c_3=c_3(a,b,\vartheta)>0$.
\end{lemma}

\begin{proof}
 Fix $j$. We are working in the almost surely invertible case fixed at the
 start of this section, so every subcollection of columns is linearly
 independent. Thus $H_j^{\mathrm{ext}}$ has codimension exactly
 $m_l+l\le (3+2\vartheta)l$ almost surely.

 Conditionally on the sigma-field generated by the columns
 $\{A^r:r>m_l,\ r\notin \mathcal I_j\}$, the subspace $H_j^{\mathrm{ext}}$ is
 deterministic and the random vectors $\{A^k:k\in\mathcal I_j\}$ are
 independent. Therefore, for every $k\in\mathcal I_j$,
 \[
   \EE\!\left[\widetilde{Z}_k^2 \,\middle|\, H_j^{\mathrm{ext}}\right]
   = \EE\!\left[\|P_{(H_j^{\mathrm{ext}})^\perp}A^k\|_2^2
     \,\middle|\, H_j^{\mathrm{ext}}\right]
   = \tr\!\bigl(P_{(H_j^{\mathrm{ext}})^\perp}\bigr)
   = m_l+l
   \le (3+2\vartheta)l.
 \]
	 By Markov's inequality,
 \[
   \bbP\!\left(
     \widetilde{Z}_k^2 > 4(3+2\vartheta)l
     \,\middle|\,
     H_j^{\mathrm{ext}}
   \right)
   \le \frac14
   \qquad\text{almost surely}.
 \]
	 Choose $C_2 := 2\sqrt{3+2\vartheta}$. Then each indicator
 $\one_{\{\widetilde Z_k \le C_2\sqrt l\}}$ has conditional mean at least
 $3/4$. Since these indicators are conditionally independent given
 $H_j^{\mathrm{ext}}$, a Chernoff bound gives
 \[
   \bbP\!\left(\mathcal E_j^c \,\middle|\, H_j^{\mathrm{ext}}\right)
   \le e^{-c_3 l}.
 \]
 Therefore
 \[
   \bbP(\mathcal E_j^c)
   \le e^{-c_3 l},
 \]
 which is the required estimate after renaming constants.
\end{proof}

\begin{lemma}[Lower bound for $\|B\|_{HS}$, uniform over choices of $A_l$]
\label{lem:BHS-lower}
There exist constants $c,C>0$, depending only on $a$ and $b$, such that
\[
  \bbP\!\left(\sum_{k=m_l+1}^n\dist(A^k,H_{m_l,k})^{-2}
    < c\,\frac{n}{l}\right)
  \;\le\;e^{-Cl}.
\]
In particular,
\[
  \bbP\!\left(\|B\|_{HS} < \sqrt{\frac{cn}{l}}\right)
  \;\le\;\exp(-C l),
\]
and this bound holds for every choice of $S\subseteq[m_l]$ with $|S|=l$.
\end{lemma}

\begin{proof}
  Combine \eqref{eq:BHS-lower-indep} and \eqref{eq:Ztilde-lower} to get
\[
  \|B\|_{HS}^2
  \;\ge\;
  \sum_{j=1}^m\sum_{k\in\mathcal{I}_j}\widetilde{Z}_k^{-2}.
\]
 On $\mathcal{E}_j$, at least $l/2$ indices in $\mathcal{I}_j$ satisfy
 $\widetilde{Z}_k\le C_2\sqrt{l}$, and for those indices
 $\widetilde{Z}_k^{-2}\ge(C_2^2 l)^{-1}$. Hence
 \[
   \sum_{k\in\mathcal{I}_j}\widetilde{Z}_k^{-2}
   \;\ge\;
   \frac{l}{2}\cdot \frac{1}{C_2^2 l}
   = \frac{1}{2C_2^2}.
 \]
 Therefore
 \[
   \|B\|_{HS}^2
   \;\ge\;
   \frac{1}{2C_2^2}\sum_{j=1}^m\mathbf{1}_{\mathcal{E}_j}.
 \]
 Let
 \[
   Z:=\sum_{j=1}^m \mathbf 1_{\mathcal E_j^c}.
 \]
 The block estimate gives $\EE Z\le me^{-c_4l}$, without requiring
 independence between the block events. Hence Markov's inequality yields
 \[
   \bbP\!\left(Z\ge \frac m2\right)
   \le \frac{2\EE Z}{m}
   \le 2e^{-c_4l}.
 \]
 Thus, with probability at least $1-2e^{-c_4l}$, at least $m/2$ blocks are
 good. On this event,
\[
  \|B\|_{HS}^2
  \;\ge\;
  \frac{m}{4C_2^2}
  \;\ge\;
   \frac{c_\vartheta' n}{4C_2^2 l}
  \;=:\;
  c\,\frac{n}{l},
\]
 giving the first inequality, since $m\ge c_\vartheta' n/l$.

 Since the entire lower bound uses only the columns $\{A^k:k>m_l\}$, it is
 independent of the choice of $S\subseteq[m_l]$ and therefore holds uniformly
 over all choices of $A_l$.
\end{proof}

\begin{lemma}[Hilbert--Schmidt upper bound for $B$]\label{lem:FW_HS_B}
There exist constants $C,c,L>0$, depending only on $a$ and $b$, such that for
every $4\le l\le c_0n$,
\begin{equation}\label{eq:B_HS_upper-linear}
  \bbP\!\left(
    \exists\,S\subseteq[m_l],\ |S|=l:
    \|B_S\|_{HS}>C\sqrt{\frac nl}
  \right)
  \le e^{-cl}.
\end{equation}
Moreover, in the same range, if $l\le L\log(2n)$, then for every $t\ge1$ and
every fixed $S\subseteq[m_l]$, $|S|=l$, one has
\begin{equation}\label{eq:B_HS_upper}
  \bbP\!\left(\|B\|_{HS} > t\sqrt{\frac{n}{l}}\right)
  \;\le\;
  \left(\frac{C}{t}\right)^l + e^{-cl}.
\end{equation}
\end{lemma}

\begin{proof}
We consider the two ranges separately. Assume first $l\le L\log(2n)$, where
$L>0$ is fixed. Fix an index set $S \subseteq [m_l]$ with $|S| = l$ and set
$T := [n] \setminus S$. By Proposition~\ref{prop:biorthogonal} and
Lemma~\ref{lem:block-biorthogonal},
\begin{equation}\label{eq:BHS-id-upper}
  \|B\|_{HS}^2
  \;=\; \sum_{k \in T} \dist(A^k, H_{l,k})^{-2},
\end{equation}
where $H_{l,k} := \Span(A^j : j \in T,\, j \ne k)$. For each
$k \in T$, the columns generating $H_{l,k}$ are indexed by
$T \setminus \{k\}$, which is disjoint from $\{k\}$, so $A^k$ is independent
of $H_{l,k}$.

\medskip
\noindent\textit{Truncation and weak $L^{l/2}$ estimates.}
Recall that the weak $L^p$ quasi-norm of a real random variable $Z$ is
defined by
\[
  \|Z\|_{p,\infty}
  :=
  \sup_{u > 0} u\,(\bbP(|Z| > u))^{1/p}.
\]
For $p \ge 2$ this satisfies the triangle inequality up to an absolute
constant,
\[
  \left\|\sum_i Z_i\right\|_{p,\infty}
  \le C \sum_i \|Z_i\|_{p,\infty},
\]
see, e.g., \cite[Theorem~3.21]{SteinWeiss}.

Fix a truncation level $u_0 := C_0 l / n$ with $C_0 > 0$ a small constant,
depending only on $a$ and $b$, to be chosen, and define
\[
  W_k
  \;:=\;
  \min\!\left(
    \dist(A^k, H_{l,k})^{-2},\;
    (u_0\sqrt{l})^{-2}
  \right),
  \qquad k \in T.
\]
The subspace $H_{l,k}$ is generated by $n-l-1$ columns, so the relevant
codimension is $l+1$. Since $l\le L\log(2n)$, the condition
$l+1\le \lambda_F n/\log n$ in Theorem~\ref{distance_theorem} holds for all
large $n$; the finitely many remaining values of $n$ are absorbed into the
constants. Thus, for every $k \in T$ and every $\tau > 0$,
\begin{equation}\label{eq:dist-upper}
  \bbP\!\left(\dist(A^k, H_{l,k}) < \tau\sqrt{l}\right)
  \;\le\; (C\tau)^l + e^{-cn}.
\end{equation}
We bound the weak $L^{l/2}$ quasi-norm of each $W_k$. Setting
$p = l/2$ and $s = (\tau\sqrt{l})^{-2} = 1/(\tau^2 l)$ so that
$\{W_k > s\} \subseteq \{\dist(A^k, H_{l,k}) < \tau\sqrt{l}\}$ for
$\tau \ge u_0$, we compute:
\begin{align}
  \|W_k\|_{l/2,\infty}
  &= \sup_{\tau \ge u_0}
    \frac{1}{\tau^2 l}
    \bigl(\bbP(W_k > (\tau\sqrt{l})^{-2})\bigr)^{2/l} \notag \\
  &\le \sup_{\tau \ge u_0}
    \frac{1}{\tau^2 l}
    \left[
      (C\tau)^2 + e^{-2cn/l}
    \right] \notag \\
  &\le \sup_{\tau \ge u_0}
    \left[
      \frac{C^2}{l} + \frac{e^{-2cn/l}}{\tau^2 l}
    \right]. \label{eq:Wk-weak}
\end{align}
Since $\tau \ge u_0 = C_0 l/n$, we have $\tau^{-2} \le n^2/(C_0^2 l^2)$, and therefore
\[
  \frac{e^{-2cn/l}}{\tau^2 l}
  \;\le\;
  \frac{n^2}{C_0^2 l^3} e^{-2cn/l}.
\]
Since $l\le L\log(2n)$, the right-hand side is bounded by $C/l$ uniformly in
$n$ and $l$ after increasing $C$ if necessary. Therefore
\eqref{eq:Wk-weak} yields
\begin{equation}\label{eq:Wk-weak-final}
  \|W_k\|_{l/2,\infty} \;\le\; \frac{C}{l}
  \qquad \text{for all } k \in T.
\end{equation}
Since $|T| = n - l \le n$, the weak triangle inequality gives
\[
  \left\|\sum_{k \in T} W_k\right\|_{l/2,\infty}
  \;\le\; C \sum_{k \in T} \|W_k\|_{l/2,\infty}
  \;\le\; \frac{Cn}{l}.
\]
By the definition of the weak norm, this implies that for every $t \ge 1$,
\begin{equation}\label{eq:sumW-tail}
  \bbP\!\left(\sum_{k \in T} W_k > t^2 \frac{n}{l}\right)
  \;\le\; \left(\frac{C}{t}\right)^l.
\end{equation}

\medskip
\noindent\textit{Truncation error.}
The truncated and untruncated quantities differ only when
\[
  \dist(A^k,H_{l,k})<u_0\sqrt l.
\]
By \eqref{eq:dist-upper} with $\tau=u_0=C_0l/n$ and a union bound over
$k\in T$,
\[
  \bbP\!\left(
    \exists\, k \in T :\;
    W_k \ne \dist(A^k, H_{l,k})^{-2}
  \right)
  \le
  (n-l)\left[\left(CC_0\frac{l}{n}\right)^l + e^{-cn}\right].
\]
We claim that the right-hand side is bounded by $e^{-c'l}$ for suitable
constants $C_0,c'>0$. The second term is immediate, since
\[
  (n-l)e^{-cn}\le ne^{-cn}\le e^{-c'l}
\]
after decreasing $c'>0$ if necessary. For the first term, set
\[
  F(l):=n\left(CC_0\frac{l}{n}\right)^l.
\]
Then
\[
  \log F(l)=\log n + l\log\!\left(CC_0\frac{l}{n}\right),
\]
so
\[
  \frac{d}{dl}\log F(l)
  =
  \log\!\left(CC_0\frac{l}{n}\right)+1.
\]
Since $l\le L\log(2n)$, choosing $C_0>0$ sufficiently small ensures, for all
large $n$,
\[
  CC_0\frac{l}{n}\le CC_0\frac{L\log(2n)}{n}\le e^{-3},
\]
and hence
\[
  \frac{d}{dl}\log F(l)\le -2
\]
throughout the admissible range of $l$. Integrating from $1$ to $l$ gives
\[
  \log F(l)\le \log F(1)-2(l-1),
\]
and hence
\[
  F(l)\le F(1)e^{-2(l-1)}=CC_0\,e^{-2(l-1)}.
\]
After shrinking $C_0$ further, this yields
\[
  n\left(CC_0\frac{l}{n}\right)^l
  \le
  e^{-c'l}.
\]
Combining the two estimates, we obtain
\begin{equation}\label{eq:trunc-error}
  \bbP\!\left(
    \exists\, k \in T :\;
    W_k \ne \dist(A^k, H_{l,k})^{-2}
  \right)
  \;\le\; e^{-c'l}.
\end{equation}

\medskip
\noindent\textit{Conclusion.}
On the event that $W_k = \dist(A^k, H_{l,k})^{-2}$ for all $k \in T$, we have
$\|B\|_{HS}^2 = \sum_{k \in T} W_k$. Combining
\eqref{eq:sumW-tail} and \eqref{eq:trunc-error}:
\[
  \bbP\!\left(\|B\|_{HS}^2 > t^2 \frac{n}{l}\right)
  \;\le\;
  \left(\frac{C}{t}\right)^l + e^{-cl},
\]
which gives \eqref{eq:B_HS_upper} upon taking square roots.

\medskip
\noindent\textit{Uniformity over $A_l$.}
For each fixed $S$, the argument above is valid since for every $k \in T$ the
column $A^k$ is independent of $H_{l,k}$ regardless of the choice of $S$.
Applying a union bound over the $\binom{m_l}{l}\le e^{c_\vartheta l}$ choices
of $S \subseteq [m_l]$ gives
\[
  \bbP\!\left(
    \exists\, S\subseteq[m_l],\ |S|=l :
    \|B_S\|_{HS} > t\sqrt{\frac{n}{l}}
  \right)
  \;\le\;
  e^{c_\vartheta l}\left[\left(\frac{C}{t}\right)^l + e^{-cl}\right].
\]
Choosing $t$ to be a sufficiently large constant and taking
$c_\vartheta$ small proves \eqref{eq:B_HS_upper-linear} for
$l\le L\log(2n)$.

\medskip\noindent
It remains to consider $L\log(2n)\le l\le c_0n$. Fix deterministic
$S\subseteq[m_l]$, $|S|=l$, and $k\in T=[n]\setminus S$. Put
\[
  Q_{S,k}:=P_{H_{l,k}^\perp}.
\]
Condition on all columns except $A^k$. Then $Q_{S,k}$ is deterministic and
independent of $A^k$. Since $H_{l,k}$ is generated by $n-l-1$ columns,
\[
  \rank(Q_{S,k})\ge l+1.
\]
Applying Proposition~\ref{prop:one-scale-anisotropic} with
$D=Q_{S,k}$ and $Z=A^k$, and using
$\|Q_{S,k}\|=1$ and $\|Q_{S,k}\|_{HS}^2=\rank(Q_{S,k})$, we obtain
\[
  \PP\!\left(
    \dist(A^k,H_{l,k})\le \eta\sqrt l
    \,\middle|\, H_{l,k}
  \right)
  \le 2e^{-c_al}.
\]
Therefore
\[
  \PP\!\left(
    \exists\,S\subseteq[m_l],\ |S|=l,\ \exists\,k\notin S:
    \dist(A^k,H_{l,k})\le \eta\sqrt l
  \right)
  \le
  2n\binom{m_l}{l}e^{-c_al}.
\]
Choose $\vartheta$ so small that $c_\vartheta<c_a/4$, and then choose
$L>4/c_a$. Since $l\ge L\log(2n)$ in the present range,
$\log(2n)\le c_al/4$. Hence
\[
  2n\binom{m_l}{l}e^{-c_al}
  \le
  \exp\!\left(-\frac{c_a}{2}l\right).
\]
On the complementary event,
\[
  \|B_S\|_{HS}^2
  =
  \sum_{k\in T}\dist(A^k,H_{l,k})^{-2}
  \le
  \frac n{\eta^2l}
\]
simultaneously for all $S$. This proves \eqref{eq:B_HS_upper-linear}.
\end{proof}

\medskip\noindent
\textit{Upper operator-norm bound for \(B\).}

Fix any deterministic index set $S\subseteq[m_l]$ with $|S|=l$, and let
$T=[n]\setminus S$. By Lemma~\ref{lem:block-biorthogonal}, adapted to the
partition $(S,T)$:
\begin{equation}\label{eq:B-opnorm}
  \|B_S\|^2
  \;=\;
  s_{\min}(A_T)^{-2},
\end{equation}
where $A_T:=(A^j)_{j\in T}$ is the $n\times(n-l)$ submatrix of columns indexed
by $T$, and $s_{\min}(A_T)$ denotes its smallest singular value.

The matrix $A_T$ has $n-l$ independent columns, each a vector in $\R^n$ whose
entries satisfy \textnormal{(A1)--(A2)}. We apply
Theorem~\ref{smallest_sing_val_lower} to $A_T$ with $N=n$ and
$n_{\mathrm{cols}}=n-l$: for all sufficiently small $\varepsilon>0$,
\[
  \bbP\!\left(s_{\min}(A_T) \le \varepsilon\!\left(\sqrt{n+1}-\sqrt{n-l}\,\right)\right)
  \;\le\;\bigl(C\varepsilon\log(1/\varepsilon)\bigr)^{l+1}+e^{-cn}.
\]
For $l\le n/2$,
\[
  \sqrt{n+1}-\sqrt{n-l}\ge \frac{l}{2\sqrt n}.
\]
Taking $\varepsilon=\alpha^{-1}$, with $\alpha$ chosen large enough, gives
\[
  \bbP\!\left(s_{\min}(A_T) \le \frac{l}{2\alpha\sqrt{n}}\right)
  \;\le\;
  \left(\frac{C\log\alpha}{\alpha}\right)^{l+1}+e^{-cn}.
\]
On the complementary event, $s_{\min}(A_T)\ge l/(2\alpha\sqrt{n})$, and by
\eqref{eq:B-opnorm}:
\[
  \|B_S\| \;\le\; \frac{2\alpha\sqrt{n}}{l}.
\]

Since there are $\binom{m_l}{l}\le e^{c_\vartheta l}$ choices of $S$, a union
bound over all choices gives
\begin{equation}\label{eq:Bopnorm-unionbound}
  \bbP\!\left(\exists\,S\subseteq[m_l],\,|S|=l:\;\|B_S\|>\frac{2\alpha\sqrt{n}}{l}\right)
  \;\le\;
  e^{c_\vartheta l}\!\left[\left(\frac{C\log\alpha}{\alpha}\right)^{l+1}+e^{-cn}\right].
\end{equation}
Choose $\alpha$ sufficiently large so that $C\log\alpha/\alpha\le e^{-2}$.
Later we take $\vartheta$ small enough that $c_\vartheta\le1$; then the first
term is at most $e^{-c'l}$, and the second is negligible. After renaming
\(2\alpha\) as \(\alpha\) and adjusting constants, we conclude that
\begin{equation}\label{eq:step122-conclusion}
  \bbP\!\left(
    \exists\,S\subseteq[m_l],\, |S|=l:
    \|B_S\| > \frac{\alpha\sqrt{n}}{l}
  \right)
  \;\le\;
  \left(\frac{C\log\alpha}{\alpha}\right)^l + e^{-cn}.
\end{equation}

\begin{remark}
The correct scale for $\|B\|$ is $\alpha\sqrt{n}/l$, not $\alpha\sqrt{n/l}$.
This distinction is essential for the small-ball estimate below: the stable
rank of $B$ requires $\|B\|^2\sim n/l^2$ in the denominator.
\end{remark}

Define the global bad event
\begin{equation}\label{eq:E2}
  \begin{aligned}
  \cE_2
  :=
  \bigl\{A:\;&
    \exists\, S\subseteq[m_l],\ |S|=l \text{ such that }\\
  &
    \|B_S\|_{HS} < \sqrt{\frac{c n}{l}}
    \text{ or }
    \|B_S\|_{HS} > C\sqrt{\frac n l}
    \text{ or }
    \|B_S\| > \frac{\alpha\sqrt{n}}{l}
  \bigr\}.
  \end{aligned}
\end{equation}

\begin{proposition}[Uniform comparison-matrix bounds]
\label{prop:comparison-matrix-bounds}
The event $\cE_2$ satisfies
\begin{equation}\label{eq:E2bound}
  \bbP(\cE_2)
  \;\le\;
  e^{-C l} + \left(\frac{C\log\alpha}{\alpha}\right)^l + e^{-cn}.
\end{equation}
Moreover, on $\cE_2^c$, every deterministic choice of
$S\subseteq[m_l]$ with $|S|=l$ satisfies
\[
  \sqrt{\frac{c n}{l}}
  \le
  \|B_S\|_{HS}
  \le
  C\sqrt{\frac n l},
  \qquad
  \|B_S\|
  \le
  \frac{\alpha\sqrt n}{l}.
\]
\end{proposition}

\begin{proof}
By Lemmas~\ref{lem:BHS-lower}, \ref{lem:FW_HS_B}, and
\eqref{eq:step122-conclusion}, each of the three alternatives in
\eqref{eq:E2} has the stated probability bound after adjusting constants.
The displayed bounds are exactly the complement of $\cE_2$.
\end{proof}

\medskip\noindent
\textit{Small ball for the deterministic block.}

For such a deterministic $S$, write
\[
  X_S(y) := A_S y,
  \qquad
  B_S := \text{the comparison matrix associated with } S.
\]
For this fixed $S$, define the local good event
\[
  \mathcal G_S
  :=
  \left\{
    \|B_S\|_{HS} \ge \sqrt{\frac{c n}{l}}
    \;\;\text{and}\;\;
    \|B_S\| \le \frac{\alpha\sqrt{n}}{l}
  \right\}.
\]
On $\mathcal G_S$, the stable rank of $B_S$ satisfies
\begin{equation}\label{eq:stable-rank}
  \sr(B_S)
  =
  \frac{\|B_S\|_{HS}^2}{\|B_S\|^2}
  \;\ge\;
  \frac{c(n/l)}{\alpha^2(n/l^2)}
  \;=\;
  \frac{c\,l}{\alpha^2}.
\end{equation}

\begin{lemma}[Pointwise comparison-image lower bound]
\label{lem:pointwise-comparison-image}
For every deterministic choice of \(S\subseteq[m_l]\) with \(|S|=l\) and every
fixed \(y\in S^{l-1}\),
\begin{equation}\label{eq:BX-smallball-aux}
  \bbP\!\left(
    \|B_S X_S(y)\|_2 \le c_\ast\sqrt{\frac{n}{l}}
    \ \text{ and }\ 
    \cE_2^c
  \right)
  \;\le\;
  e^{-c_{\mathrm{sb}} l}.
\end{equation}
\end{lemma}

\begin{proof}
For each row index $i$,
\[
  (X_S(y))_i
  =
  \sum_{j\in S} y_j\xi_{ij}.
\]
The summands are independent, have mean zero and variance one, and satisfy the
uniform anti-concentration hypothesis. By
Lemma~\ref{lem:scalar-linear-smallball}, there exist constants
\[
  \rho_0=\rho_0(a,b)>0,
  \qquad
  \kappa_0=\kappa_0(a,b)<1,
\]
such that
\[
  \mathcal L\bigl((X_S(y))_i,\rho_0\bigr)\le \kappa_0
  \qquad\text{for every }i.
\]
Moreover, the coordinates of $X_S(y)$ are independent and satisfy
$\mathbb E (X_S(y))_i^2=1$. They need not be subgaussian; this small-ball
bound is the input used in place of the subgaussian estimates from
\cite[Theorems~2.4 and~2.5]{Feng_Wei_intermediate}.

Condition on the columns indexed by $T=[n]\setminus S$. Then $B_S$ is
deterministic and independent of $X_S(y)$. Applying
Proposition~\ref{prop:one-scale-anisotropic} with $D=B_S$ and $Z=X_S(y)$ gives
\[
  \mathbb P\!\left(
    \|B_SX_S(y)\|_2
    \le
    \eta\|B_S\|_{HS}
    \,\middle|\,
    B_S
  \right)
  \le
  2\exp\!\left(-c\,\sr(B_S)\right),
\]
where $\eta,c>0$ depend only on $a,b$. On the event $\mathcal G_S$, this
implies
\[
  \mathbb P\!\left(
    \|B_SX_S(y)\|_2
    \le
    \eta\sqrt{\frac{c n}{l}}
    \,\middle|\,
    B_S
  \right)\mathbf 1_{\mathcal G_S}
  \le
  2\exp\!\left(-c_2\frac{l}{\alpha^2}\right)\mathbf 1_{\mathcal G_S}.
\]
Choose the constant in \eqref{eq:BX-smallball-aux} so that
\[
  c_\ast\le \eta\sqrt{c}.
\]
Taking expectations yields
\[
  \mathbb P\!\left(
    \|B_SX_S(y)\|_2
    \le
    c_\ast\sqrt{\frac{n}{l}}
    \ \text{and}\ 
    \mathcal G_S
  \right)
  \le
  2\exp\!\left(-c_2\frac{l}{\alpha^2}\right).
\]
Since $\cE_2^c\subseteq \mathcal G_S$, the same bound holds with $\cE_2^c$ in
place of $\mathcal G_S$. After decreasing the exponential constant, this
proves the lemma.
\end{proof}

Since $\alpha$ is now fixed, the exponent $c_{\mathrm{sb}}>0$ in
Lemma~\ref{lem:pointwise-comparison-image} is fixed. We choose the
oversampling parameter
$\vartheta=\vartheta(a,b)>0$ small enough that
\[
  c_\vartheta < \frac{c_{\mathrm{sb}}}{4},
\]
and so that the earlier restrictions on \(c_\vartheta\) remain valid. Since
there are at most $\binom{m_l}{l}\le e^{c_\vartheta l}$ possible choices of
$S$, a union bound over all deterministic choices of $S$ shows that for the
selected block,
\begin{equation}\label{eq:BX-smallball}
  \bbP\!\left(
    \|B A_l y\|_2 \le c_\ast\sqrt{\frac{n}{l}}
    \ \text{ and }\ 
    \cE_2^c
  \right)
  \;\le\;
  e^{-c_{\mathrm{sb}}' l}
\end{equation}
for some $c_{\mathrm{sb}}'=c_{\mathrm{sb}}'(a,b)>0$.

\medskip\noindent
\textit{Conclusion of the pointwise step.}

Define the combined bad event
\[
  \cE_{\mathrm{bad}}(y)
  \;:=\;
  \cE_1\cup\cE_2
  \;\cup\;
  \left\{A:\;\|B A_l y\|_2 \le c_\ast\sqrt{\frac{n}{l}}\right\}.
\]
From \eqref{eq:step11-conclusion}, \eqref{eq:E2bound}, and
\eqref{eq:BX-smallball}:
\begin{equation}\label{eq:step1-total}
  \bbP(\cE_{\mathrm{bad}}(y))
  \;\le\;
  e^{-c_1 l}
  \;+\;
  e^{-C l} + \left(\frac{C\log\alpha}{\alpha}\right)^l + e^{-cn}
  \;+\;
  e^{-c_{\mathrm{sb}}' l}.
\end{equation}
For $\alpha$ a sufficiently large constant depending only on $a$ and $b$, each
term is at most $e^{-c' l}$ for some $c'>0$. Hence there exist constants
$c',C'>0$, depending only on $a$ and $b$,
such that on $\cE_{\mathrm{bad}}(y)^c$:
\begin{align}
  \|U(y)\|_2 &\;=\;\|P_l^\perp A_l y\|_2 \;\le\; C' l, \label{eq:U-upper}\\
  \|A^{-1}U(y)\|_2 &\;\ge\; \|B A_l y\|_2 \;\ge\; c'\sqrt{\frac{n}{l}},
  \label{eq:U-lower}
\end{align}
for the selected block $A_l$, with
\begin{equation}\label{eq:step1-conclusion}
  \bbP\!\left(\|U(y)\|_2 > C' l
    \;\text{ or }\;
    \|A^{-1}U(y)\|_2 < c'\sqrt{\frac{n}{l}}\right)
  \;\le\;e^{-c' l}.
\end{equation}

\subsection{Net argument for the reduced block}\label{subsec:net-reduced}

The pointwise small-ball estimate gives an \(e^{-cl}\) bound, but a net over
the full \(l\)-dimensional sphere would be too large to absorb uniformly with
the available constants. We therefore retain only
\(l'=\lfloor\lambda l\rfloor\) columns, with \(\lambda>0\) chosen small.

Let $\alpha > 1$ be fixed as above. Let
$\varepsilon\in(0,1)$ be fixed small enough for the final approximation step
below. In the logical order, the reduced-block estimate fixes the constant
$C'$ in \eqref{eq:Ainv-HS} before this choice is made. Then choose
$\lambda \in (0,1)$ after $\varepsilon$ is fixed, and set
$l' := \lfloor \lambda l \rfloor$. We first prove a net-point lower bound
uniformly over deterministic choices of $(S,J)$; after that, the approximation
argument is applied only to the selected block $A_l$ from
Subsection~\ref{subsec:pointwise-selected} and to its
sub-collections $A_{l'}\subseteq A_l$.

\subsubsection{Net and approximation}

Specializing Theorem~\ref{thm:hs-net} to dimension $l'$, we fix a deterministic
net
\[
  \mathcal{N}_\varepsilon
  \subset
  \frac32 B_2^{l'}\setminus \frac12 B_2^{l'}
\]
with the following two properties.

\medskip
\noindent Cardinality:
\[
    |\mathcal{N}_\varepsilon| \;\le\; \left(\frac{C_{\mathrm{net}}}{\varepsilon}\right)^{l'}.
\]

\noindent Hilbert--Schmidt approximation guarantee:
For every matrix $M$ with $l'$ columns and every $y \in S^{l'-1}$, there
exists $y_i=y_i(M,y)\in \mathcal{N}_\varepsilon$ such that
\begin{equation}
    \|M(y - y_i)\|_2 \;\le\; \frac{\sqrt{2}\,\varepsilon}{\sqrt{l'}}\,\|M\|_{HS},
\end{equation}
where \eqref{eq:lattice-approx} is exactly the estimate recorded in
Theorem~\ref{thm:hs-net}. The net points are not necessarily unit vectors, but
their Euclidean norms lie between $1/2$ and $3/2$.

Once a sub-collection $A_{l'} \subseteq A_l$ is fixed, write
$U_{l'}(y) := P_l^\perp A_{l'} y$ for $y \in \mathbb{R}^{l'}$. We prove
$\|A^{-1}U_{l'}(y)\|_2 \gtrsim \sqrt{n/l}$ for all
$y \in S^{l'-1}$ simultaneously. Establishing a uniform lower bound over the
annular net suffices: for every $y \in S^{l'-1}$, let
$y_i \in \mathcal{N}_\varepsilon$ be its net approximant, so that

\[
    \|A^{-1}U_{l'}(y)\|_2
    \;\ge\;
    \|A^{-1}U_{l'}(y_i)\|_2
    \;-\;
    \|A^{-1}P_l^\perp A_{l'}(y - y_i)\|_2.
\]

The second term is controlled by applying \eqref{eq:lattice-approx} with
$M = A^{-1}P_l^\perp A_{l'}$, which requires the Hilbert--Schmidt bound from
the reduced-block construction below. The first term is handled by the
pointwise estimate above and a union bound over $\mathcal{N}_\varepsilon$.
Two estimates are needed to pass from pointwise control to the full sphere:
a lower bound on the net points and a Hilbert--Schmidt bound on the reduced
map.

\subsubsection{Net-point lower bound for the inverse image}

Fix deterministic sets
\[
  S\subseteq[m_l],\quad |S|=l,
  \qquad
  J\subseteq S,\quad |J|=l',
  \qquad
  y_i\in\mathcal N_\varepsilon.
\]
Let
\[
  H_S:=\Span(A^j:j\notin S),
  \qquad
  P_S^\perp:=P_{H_S^\perp}.
\]
Let $\widetilde y_i\in\mathbb R^l$ be the vector obtained from $y_i$ by
extending it by zero on $S\setminus J$. Then
\[
  A_S\widetilde y_i=A_Jy_i,
  \qquad
  X_S(\widetilde y_i)=A_Jy_i.
\]
Moreover,
\[
  \frac12\le \|\widetilde y_i\|_2=\|y_i\|_2\le\frac32.
\]
Set $\widehat y_i:=\widetilde y_i/\|\widetilde y_i\|_2\in S^{l-1}$.
For this deterministic choice of $S$, the comparison matrix $B_S$ is
determined by the columns outside $S$ and is independent of $A_S$. Hence
\eqref{eq:BX-smallball-aux}, applied with $y=\widehat y_i$, gives
\[
  \PP\!\left(
    \|B_SA_Jy_i\|_2
    \le \frac{c_\ast}{2}\sqrt{\frac nl},
    \ \cE_2^c
  \right)
  \le e^{-c_{\mathrm{sb}}l}.
\]
Indeed, if the displayed event occurs, then
\[
  \|B_SA_S\widehat y_i\|_2
  =
  \frac{\|B_SA_Jy_i\|_2}{\|\widetilde y_i\|_2}
  \le
  c_\ast\sqrt{\frac nl}.
\]
Now apply Lemma~\ref{lem:block-biorthogonal} to the partition
$[n]=S\sqcup([n]\setminus S)$. Since $P_S^\perp$ is the projection onto the
orthogonal complement of the span of the columns outside $S$, we have
\[
  \|A^{-1}P_S^\perp A_Jy_i\|_2^2
  =
  \|\widetilde y_i\|_2^2+\|B_SA_Jy_i\|_2^2
  \ge
  \|B_SA_Jy_i\|_2^2.
\]
Thus the inverse-image lower bound can fail only if the comparison-matrix
lower bound fails. After decreasing the constant \(c>0\), we obtain
\[
  \PP\!\left(
    \|A^{-1}P_S^\perp A_Jy_i\|_2
    \le c\sqrt{\frac nl},
    \ \cE_2^c
  \right)
  \le e^{-c_{\mathrm{sb}}l}.
\]
The number of triples $(S,J,y_i)$ is at most
\[
  \binom{m_l}{l}\binom l{l'}|\mathcal N_\varepsilon|
  \le
  \exp\!\left[
    c_\vartheta l
    +\lambda l\log\frac e\lambda
    +\lambda l\log\frac{C_{\mathrm{net}}}{\varepsilon}
  \right].
\]
The parameter $\varepsilon$ has already been fixed. By the choice of
\(\vartheta\), the first term in the exponent above is at most
\(c_{\mathrm{sb}}l/4\). Now choose \(\lambda\) sufficiently small so that
\[
  \lambda \log\frac e\lambda
  +\lambda \log\frac{C_{\mathrm{net}}}{\varepsilon}
  \le \frac{c_{\mathrm{sb}}}{4}.
\]
Then the total exponent is at most \(c_{\mathrm{sb}}l/2\). Hence
\begin{equation}\label{eq:net-lb}
  \PP\!\left(
    \exists\,S,J,y_i:
    \|A^{-1}P_S^\perp A_Jy_i\|_2
    \le c\sqrt{\frac nl},
    \ \cE_2^c
  \right)
  \le e^{-cl}.
\end{equation}
Because the union bound is over all deterministic triples, this estimate also
applies after the subcollection $J(S)$ is chosen adaptively below.

\subsubsection{Hilbert--Schmidt bound for the reduced block}

We show that every choice of $A_l \subseteq A_{m_l}$ contains a sub-collection
of $l'$ columns forming $A_{l'}$ such that, if $S$ denotes the index set of
$A_l$ and $B_S$ the corresponding comparison matrix,
\begin{equation}\label{eq:BAl-HS}
    \|B_SA_{l'}\|_{HS} \;\le\; C_K\sqrt{\lambda n}.
\end{equation}
and consequently
\begin{equation}\label{eq:Ainv-HS}
    \|A^{-1}P_l^\perp A_{l'}\|_{HS} \;\le\; C'\sqrt{\lambda n}.
\end{equation}

\medskip\noindent
Fix a deterministic choice of index set $S \subseteq [m_l]$ with $|S| = l$,
let $T := [n] \setminus S$, and let $B_S$ denote the corresponding matrix $B$
from the proof notation. Recall that $B_S$ is determined by
$\{A^j : j \in T\}$, so it is independent of every column $A^j$ for
$j \in S$ (Lemma~\ref{lem:block-biorthogonal}). For each
$j \in S$, define the good event
\[
    \mathsf{G}_S(A^j)
    \;:=\;
    \left\{\|B_SA^j\|_2 \;\le\; K\|B_S\|_{HS}\right\},
\]
where $K$ is a large constant depending only on $a$ and $b$. By isotropy
($\mathbb{E}[A^j(A^j)^T] = I_n$) and independence of $A^j$ from $B_S$:
\[
    \mathbb{E}\!\left[\|B_SA^j\|_2^2 \,\Big|\, B_S\right]
    \;=\; \mathrm{tr}(B_S^T B_S)
    \;=\; \|B_S\|_{HS}^2.
\]
By Markov's inequality,
\[
    \mathbb{P}\!\left(\mathsf{G}_S(A^j)^c \,\Big|\, B_S\right)
    \;=\;
    \mathbb{P}\!\left(
      \|B_SA^j\|_2^2 > K^2\|B_S\|_{HS}^2
      \,\Big|\, B_S
    \right)
    \;\le\; K^{-2}
    \;=:\; \eta.
\]
Thus $\eta = K^{-2}$ can be made arbitrarily small by choosing $K$ large.

\medskip\noindent
Conditionally on $B_S$, the columns $\{A^j\}_{j \in S}$ are mutually
independent, though not necessarily identically distributed. Hence the events
$\{\mathsf{G}_S(A^j)\}_{j \in S}$ are conditionally independent given $B_S$,
each with conditional probability at least $1-\eta$. The count
\[
    N_S := \#\{j \in S : \mathsf{G}_S(A^j)\}
\]
therefore stochastically dominates $\mathrm{Bin}(l,1-\eta)$ given $B_S$. By
the Chernoff bound,
\begin{equation}\label{eq:chernoff}
    \mathbb{P}\!\left(N_S < l/2 \,\Big|\, B_S\right)
    \;\le\; e^{-c_3(\eta) l}
\end{equation}
for a positive constant $c_3(\eta)$. Choose $K$ so large that
$c_3(\eta) > 2c_\vartheta$, and abbreviate $c_3 := c_3(\eta)$. Averaging
\eqref{eq:chernoff} over $B_S$ gives the unconditional estimate
\[
    \mathbb{P}(N_S < l/2) \;\le\; e^{-c_3 l}.
\]
Since $l' = \lfloor\lambda l\rfloor \le l/2$ for $\lambda \le 1/2$, on the
complementary event there are at least $l'$ good columns in $S$.

\medskip\noindent
We again take a union bound, this time over all deterministic choices of
$S \subseteq [m_l]$ with
$|S| = l$:
\begin{equation}\label{eq:union-good}
    \mathbb{P}\!\left(
      \exists\, S \subseteq [m_l],\ |S| = l \text{ such that } N_S < l/2
    \right)
    \;\le\; e^{c_\vartheta l} \cdot e^{-c_3 l}
    \;\le\; e^{-c_4 l},
\end{equation}
where $c_4 := c_3 - c_\vartheta > 0$.

For every deterministic $S\subseteq[m_l]$ with $|S|=l$, define $J(S)$ on all
outcomes as follows. If $S$ contains at least $l'$ good columns, choose
$J(S)\subseteq\{j\in S:\mathsf G_S(A^j)\}$ with $|J(S)|=l'$. Otherwise choose
an arbitrary $l'$-element subset of $S$. Set $A_{l'} := (A^j)_{j\in J(S)}$.
On the complement of the event in \eqref{eq:union-good}, this choice consists
only of good columns for every admissible $S$.

\medskip\noindent
On this event, for every $j \in J(S)$:
\[
    \|B_SA^j\|_2
    \;\le\; K\|B_S\|_{HS}.
\]
Therefore,
\begin{equation}\label{eq:BAl-via-BHS}
    \|B_SA_{l'}\|_{HS}^2
    \;=\;
    \sum_{j \in J(S)}\|B_SA^j\|_2^2
    \;\le\;
    l' \cdot K^2\|B_S\|_{HS}^2.
\end{equation}
By Lemma~\ref{lem:FW_HS_B}, $\|B_S\|_{HS} \le C\sqrt{n/l}$ with probability at
least $1 - e^{-cl}$, uniformly over all choices of $A_l \subseteq A_{m_l}$.
Since $l'\le \lambda l$, \eqref{eq:BAl-via-BHS} gives
\[
    \|B_SA_{l'}\|_{HS}^2
    \;\le\; \lambda l \cdot K^2 \cdot C^2\frac{n}{l}
    \;=\; C_K \lambda n,
\]
which establishes \eqref{eq:BAl-HS}. Using
Lemma~\ref{lem:block-biorthogonal}, for each $j \in J(S)$ we have
\[
  \|A^{-1}P_l^\perp A_{l'} e_j\|_2^2 = 1 + \|B_SA^j\|_2^2.
\]
Here \(e_j\) denotes the coordinate vector corresponding to the selected
column \(A^j\) in the domain of \(A_{l'}\).
Therefore
\begin{equation}\label{eq:Ainv-HS-derived}
    \|A^{-1}P_l^\perp A_{l'}\|_{HS}^2
    \;=\; l' + \|B_SA_{l'}\|_{HS}^2
    \;\le\; \lambda l + C_K \lambda n
    \;\le\; C'\lambda n,
\end{equation}
which gives \eqref{eq:Ainv-HS}. We do not need a bound of order $\sqrt{n/l}$
here. The approximation error carries an additional factor
$1/\sqrt{l'} \sim 1/\sqrt{l}$, so the bound
$\|A^{-1}P_l^\perp A_{l'}\|_{HS} \lesssim \sqrt{n}$ is enough.

\subsubsection{Conclusion of the net argument}

Define the bad events
\begin{equation}\label{eq:E3E4}
\begin{aligned}
    \cE_3 &:=
    \left\{A : \|A^{-1}P_l^\perp A_{l'}\|_{HS} \ge 2\sqrt{C'\lambda n}\right\}
    \cup
    \left\{A : \|P_l^\perp A_{l'}\|_{HS} \ge Cl\right\}
    \cup \cE_1 \cup \cE_2,\\
    \cE_4 &:=
    \left\{A : \exists\, y_i \in \mathcal{N}_\varepsilon
    \;\text{such that}\;
    \|A^{-1}U_{l'}(y_i)\|_2
    \le \frac{c}{2}\sqrt{\frac{n}{l}}\right\}.
\end{aligned}
\end{equation}
The estimates above give
\begin{equation}\label{eq:E3-bound}
    \mathbb{P}(\cE_3) \;\le\; e^{-c l} + e^{-c_4 l} + \mathbb{P}(\cE_2) + \mathbb{P}(\cE_1).
\end{equation}
By the deterministic-triples bound \eqref{eq:net-lb}, applied to the selected
triple after $J(S)$ is chosen,
\begin{equation}\label{eq:E4-bound}
    \mathbb{P}(\cE_4\cap\cE_2^c)
    \;\le\; e^{-c l}.
\end{equation}
After the earlier choice of $\alpha$ and after decreasing $c>0$, the bound
\eqref{eq:E3-bound} gives $\bbP(\cE_3)\le Ce^{-cl}$. Since
$\cE_2\subseteq\cE_3$, the event
\[
  \Omega_l:=(\cE_3\cup\cE_4)^c
\]
satisfies
\begin{equation}\label{eq:Omega-bound}
  \bbP(\Omega_l^c)
  \le
  C e^{-c l}.
\end{equation}

\medskip\noindent
Assume $\Omega_l$ and fix any $y \in S^{l'-1}$.
Let $y_i \in \mathcal{N}_\varepsilon$ be its net approximant. Applying
\eqref{eq:lattice-approx} with $M = A^{-1}P_l^\perp A_{l'}$ and using
$l'\ge \lambda l/2$:
\[
    \|A^{-1}U_{l'}(y) - A^{-1}U_{l'}(y_i)\|_2
    \;\le\;
    \frac{\sqrt{2}\,\varepsilon}{\sqrt{l'}}\,\|A^{-1}P_l^\perp A_{l'}\|_{HS}
    \;\le\;
    \frac{\sqrt{2}\,\varepsilon}{\sqrt{l'}} \cdot 2\sqrt{C'\lambda n}
    \;\le\; 4\sqrt{C'}\,\varepsilon\sqrt{\frac{n}{l}}.
\]
By the choice of $\varepsilon$, $4\sqrt{C'}\,\varepsilon < c/4$.
Since $\|A^{-1}U_{l'}(y_i)\|_2 > \frac{c}{2}\sqrt{n/l}$ on
$\cE_4^c$, we obtain
\begin{equation}\label{eq:step2-conclusion}
    \|A^{-1}U_{l'}(y)\|_2
    \;\ge\; \frac{c}{2}\sqrt{\frac{n}{l}} - \frac{c}{4}\sqrt{\frac{n}{l}}
    \;=\; \frac{c}{4}\sqrt{\frac{n}{l}}
    \qquad \text{for all } y \in S^{l'-1}.
\end{equation}
On the same event $\Omega_l$, the definition of $\cE_3$ gives
\begin{equation}\label{eq:U-HS-good}
  \|P_l^\perp A_{l'}\|_{HS}\le Cl.
\end{equation}

\subsection{Completion of the fixed-scale proof}

We now combine the preceding estimates to prove
Proposition~\ref{prop:fixed-scale}.

\subsubsection{Passage to singular values}

Set
\[
  U:=P_l^\perp A_{l'}:\mathbb R^{l'}\to\mathbb R^n.
\]
On $\Omega_l$, \eqref{eq:step2-conclusion} and \eqref{eq:U-HS-good} allow us
to apply Lemma~\ref{lem:singular-subspace-reduction} with
\[
  D=A,\qquad
  a=\frac{c}{4}\sqrt{\frac nl},\qquad
  H=Cl,\qquad
  m=l'.
\]
Since $l'\ge\lambda l/2$, there is a subspace
$E_{l'}\subseteq H_l^\perp$ of dimension
\[
  d_{l'}:=\left\lceil\frac{l'}{2}\right\rceil
\]
such that
\[
  \min_{u\in S_{E_{l'}}}\|A^{-1}u\|_2
  \ge
  c_\lambda\frac{\sqrt n}{l}.
\]
By the min-max characterization of singular values,
\[
  s_{d_{l'}}(A^{-1})
  =
  \max_{\dim F=d_{l'}}\ \min_{u\in S_F}\|A^{-1}u\|_2
  \ge
  \min_{u\in S_{E_{l'}}}\|A^{-1}u\|_2
  \ge
  c_\lambda\frac{\sqrt n}{l}.
\]
Equivalently,
\begin{equation}\label{eq:sd-A-upper}
  s_{n+1-d_{l'}}(A)\le C_\lambda\frac{l}{\sqrt n}.
\end{equation}
We have therefore shown that there exist constants $c_9,c_{10},C_9,C_{10}>0$ such
that for every proof parameter $m$ satisfying
\[
  \frac{2}{\lambda}\le m \le c_{10}n,
\]
one has
\begin{equation}\label{eq:intermediate-m}
  \bbP\!\left(
    s_{n+1-d_m}(A)
    \ge C_{10}\frac{m}{\sqrt{n}}
  \right)
  \;\le\; C_9e^{-c_9 m}.
\end{equation}
Here
\[
  d_m:=\left\lceil\frac{\lfloor\lambda m\rfloor}{2}\right\rceil.
\]
This follows from \eqref{eq:sd-A-upper}, with $m$ in place of the proof
parameter $l$, together with \eqref{eq:Omega-bound}.

\medskip\noindent
It remains to derive Proposition~\ref{prop:fixed-scale} from
\eqref{eq:intermediate-m}. Fix a target index $l$ satisfying
\[
  1 \le l \le c_0 n,
\]
where $c_0 > 0$ will be chosen shortly, and set
\[
  m := \left\lceil \frac{2l}{\lambda} \right\rceil.
\]
Then
\[
  m\ge \frac{2}{\lambda},
  \qquad
  \lfloor \lambda m \rfloor \ge 2l-1,
  \qquad\text{so}\qquad
  d_m=\left\lceil\frac{\lfloor\lambda m\rfloor}{2}\right\rceil\ge l.
\]
By monotonicity of singular values,
\[
  s_{n+1-l}(A) \le s_{n+1-d_m}(A).
\]
If $c_0$ is chosen so that $(2/\lambda+1)c_0 \le c_{10}$, then
$m \le c_{10} n$, so \eqref{eq:intermediate-m} applies. Since also
$m \le (2/\lambda + 1)l \le C_\lambda l$, we obtain
\[
  \left\{
    s_{n+1-l}(A) \ge C_{10} C_\lambda\frac{l}{\sqrt{n}}
  \right\}
  \subseteq
  \left\{
    s_{n+1-d_m}(A)
    \ge C_{10}\frac{m}{\sqrt{n}}
  \right\},
\]
and therefore
\begin{equation}\label{eq:claim-proof-small-l}
  \bbP\!\left(s_{n+1-l}(A) \ge C_1\frac{l}{\sqrt{n}}\right)
  \;\le\; C_9e^{-c_9 m}
  \;\le\; e^{-c_{11} l}
  \qquad\text{for all }1 \le l \le c_0 n,
\end{equation}
with constants $C_1, c_{11} > 0$ depending only on $a$, $b$, and $\lambda$.
Here $\lambda$ was chosen small enough, after the preceding constants were
fixed, to absorb the fixed prefactor $C_9$ into the exponential. This proves
Proposition~\ref{prop:fixed-scale}.

\section{Bulk Completion and Consequences}
\label{sec:rescaling-consequences}

This section proves the remaining fixed-proportion estimate and then derives
Theorem~\ref{thm:main} and the corollaries.

\subsection{A bulk estimate by column trimming}
\label{subsec:bulk-trimming}

We use the following elementary estimate in the fixed-proportion range.

\begin{lemma}[Bulk bound by column trimming]
\label{lem:bulk-trimming}
Let
\[
  M=[X_1,\dots,X_n]
\]
be an $n\times n$ random matrix with independent columns satisfying
\[
  \EE\|X_j\|_2^2\le n,
  \qquad j\in[n].
\]
Then, for every $\delta\in(0,1)$, there exist absolute constants $c,C>0$ such
that
\[
  \bbP\!\left(
    s_{\lceil\delta n\rceil}(M)
    > C\delta^{-1/2}\sqrt n
  \right)
  \le
  \exp(-c\delta n).
\]
Consequently, on the same event,
\[
  s_k(M)\le C\delta^{-1/2}\sqrt n
  \qquad
  \text{for every }k\ge\lceil\delta n\rceil.
\]
\end{lemma}

Deleting \(r\) columns can shift the singular-value index by at most \(r\),
while the remaining Hilbert--Schmidt mass controls all singular values below
that shifted index.

\begin{proof}
Fix $\delta\in(0,1)$ and set
\[
  L:=\frac{8}{\delta}.
\]
Define
\[
  \mathcal B:=\{j\in[n]:\|X_j\|_2^2>Ln\},
  \qquad
  r:=|\mathcal B|.
\]
Since $\EE\|X_j\|_2^2\le n$, Markov's inequality gives
\[
  \bbP(j\in\mathcal B)\le \frac1L=\frac{\delta}{8}.
\]
The indicators $\mathbf 1_{\{j\in\mathcal B\}}$ are independent. Therefore,
by a Chernoff bound,
\begin{equation}\label{eq:number-long-columns}
  \bbP\!\left(r>\frac{\delta n}{4}\right)
  \le
  \exp(-c\delta n).
\end{equation}

Set
\[
  Y_j:=\|X_j\|_2^2\mathbf 1_{\{\|X_j\|_2^2\le Ln\}}.
\]
Then the variables $Y_j$ are independent and satisfy
\[
  0\le Y_j\le Ln,
  \qquad
  \sum_{j=1}^n\EE Y_j\le n^2.
\]
Moreover,
\[
  \sum_{j=1}^n\operatorname{Var}(Y_j)
  \le
  \sum_{j=1}^n\EE Y_j^2
  \le
  Ln\sum_{j=1}^n\EE Y_j
  \le
  Ln^3.
\]
Bernstein's inequality yields
\begin{equation}\label{eq:trimmed-HS-bound}
  \bbP\!\left(
    \sum_{j=1}^nY_j>2n^2
  \right)
  \le
  \exp\!\left(-c\frac{n}{L}\right)
  \le
  \exp(-c\delta n).
\end{equation}

Assume that the complementary events in \eqref{eq:number-long-columns} and
\eqref{eq:trimmed-HS-bound} occur. Let
\[
  \mathcal G:=[n]\setminus\mathcal B,
\]
and let $M_{\mathcal G}$ be the submatrix obtained by retaining only the
columns indexed by $\mathcal G$. Then
\[
  r\le\frac{\delta n}{4},
  \qquad
  \|M_{\mathcal G}\|_{HS}^2
  =
  \sum_{j=1}^nY_j
  \le 2n^2.
\]
Set $k_0:=\lceil\delta n\rceil$. Since
$M_{\mathcal G}^*M_{\mathcal G}$ is a principal submatrix of $M^*M$, Cauchy
interlacing gives
\[
  s_{k_0}(M)\le s_{k_0-r}(M_{\mathcal G}).
\]
Furthermore,
\[
  k_0-r\ge\frac{3\delta n}{4}.
\]
Using $s_j(N)\le\|N\|_{HS}/\sqrt j$, we obtain
\[
  s_{k_0}(M)
  \le
  \frac{\|M_{\mathcal G}\|_{HS}}{\sqrt{k_0-r}}
  \le
  \sqrt{\frac{8}{3\delta}}\sqrt n.
\]
Monotonicity gives the same estimate for every $k\ge k_0$. The case of
$\delta n\le C_0$, where $C_0$ is the fixed constant implicit in the preceding
argument, follows from Markov's inequality. Indeed,
\[
  k_0s_{k_0}(M)^2\le\|M\|_{HS}^2,
\]
and therefore
\[
  \bbP\!\left(
    s_{k_0}(M)>C\delta^{-1/2}\sqrt n
  \right)
  \le
  \frac{\delta\,\EE\|M\|_{HS}^2}{C^2k_0n}
  \le
  \frac1{C^2}.
\]
If $\delta n$ is bounded by this fixed constant, choosing $C$ large enough
makes this at most
$\exp(-c\delta n)$.
\end{proof}

\begin{proposition}[Fixed-scale estimate away from the upper edge]
\label{prop:fixed-scale-away-edge}
For every $\delta\in(0,1)$, there exist constants
$c_\delta,C_\delta>0$, depending only on $a,b,\delta$, such that
\[
  \bbP\!\left(
    s_{n+1-l}(A)>C_\delta\frac{l}{\sqrt n}
  \right)
  \le
  \exp(-c_\delta l)
\]
for every integer $1\le l\le(1-\delta)n$.
\end{proposition}

\begin{proof}
Let $\gamma=\gamma(a,b)>0$ denote the range constant $c_0$ from
Proposition~\ref{prop:fixed-scale}.

If $l\le\gamma n$, the conclusion follows directly from
Proposition~\ref{prop:fixed-scale}. Assume now that
\[
  \gamma n<l\le(1-\delta)n
\]
and put
\[
  k:=n+1-l.
\]
Then $k\ge\lceil\delta n\rceil$. The columns of $A$ are independent and
satisfy $\EE\|A^j\|_2^2=n$. Lemma~\ref{lem:bulk-trimming} and monotonicity
therefore give, outside an event of probability at most $\exp(-c\delta n)$,
\[
  s_{n+1-l}(A)
  =
  s_k(A)
  \le
  s_{\lceil\delta n\rceil}(A)
  \le
  C\delta^{-1/2}\sqrt n.
\]
Since $l>\gamma n$,
\[
  \sqrt n
  \le
  \frac1\gamma\frac{l}{\sqrt n},
\]
and hence
\[
  s_{n+1-l}(A)
  \le
  \frac{C}{\gamma\sqrt\delta}\frac{l}{\sqrt n}.
\]
Finally, since $l\le n$,
\[
  \exp(-c\delta n)\le\exp(-c\delta l).
\]
Combining the two ranges and adjusting constants proves the proposition.
\end{proof}

\subsection{Recovering the deviation parameter by monotonicity}
\label{subsec:monotonicity-deviation}

Fix $\delta\in(0,1)$, $t\ge1$, and an integer $l$ satisfying
\[
  1\le l\le(1-\delta)n.
\]
Set
\[
  r:=\left\lfloor\min\{tl,(1-\delta)n\}\right\rfloor.
\]
Since $t\ge1$ and $l\le(1-\delta)n$, we have \(r\ge l\). Also
$1\le r\le(1-\delta)n$, so Proposition~\ref{prop:fixed-scale-away-edge}
applies at level $r$:
\[
  \bbP\!\left(
    s_{n+1-r}(A)>C_\delta\frac{r}{\sqrt n}
  \right)
  \le
  \exp(-c_\delta r).
\]
Since $l\le r$, monotonicity gives
\[
  s_{n+1-l}(A)\le s_{n+1-r}(A).
\]
Moreover, \(r\le tl\). Consequently,
\[
  \bbP\!\left(
    s_{n+1-l}(A)>C_\delta t\frac{l}{\sqrt n}
  \right)
  \le
  \exp(-c_\delta r).
\]
Since \(\min\{tl,(1-\delta)n\}\ge1\), we have
\[
  r\ge \frac12\min\{tl,(1-\delta)n\}.
\]
In addition,
\[
  \min\{tl,(1-\delta)n\}
  \ge
  (1-\delta)\min\{tl,n\}.
\]
Therefore,
\[
  \exp(-c_\delta r)
  \le
  \exp\!\left(
    -\frac{c_\delta(1-\delta)}{2}
    \min\{tl,n\}
  \right).
\]
Absorbing the factor \((1-\delta)/2\) into the constants proves
Theorem~\ref{thm:main}.

\begin{remark}
The two parts of the argument play different roles.
Proposition~\ref{prop:fixed-scale} contains the lower-edge inverse-matrix
argument, while Lemma~\ref{lem:bulk-trimming} is used only to complete the
range where $l$ is proportional to $n$. No part of the inverse-matrix
construction needs to be extended into this bulk range.
\end{remark}

\subsection{Proofs of the corollaries}

\begin{proof}[Proof of Corollary~\ref{cor:square-profile}]
The upper bound is Theorem~\ref{thm:main} with $t=1$. For the lower bound,
fix any $n\times(n-l+1)$ column submatrix $A_0$ of $A$. By Cauchy interlacing,
\[
  s_{n+1-l}(A)\ge s_{\min}(A_0).
\]
Theorem~\ref{smallest_sing_val_lower}, applied to $A_0$, gives, for all
sufficiently small $\varepsilon>0$,
\[
  \PP\!\left(
    s_{\min}(A_0)
    \le
    \varepsilon(\sqrt{n+1}-\sqrt{n-l+1})
  \right)
  \le
  (C\varepsilon\log(1/\varepsilon))^l+e^{-cn}.
\]
For every $1\le l\le n$,
\[
  \sqrt{n+1}-\sqrt{n-l+1}
  =
  \frac{l}
       {\sqrt{n+1}+\sqrt{n-l+1}}
  \asymp
  \frac l{\sqrt n}.
\]
Choosing $\varepsilon>0$ sufficiently small gives the claimed lower tail.
Combining the two estimates proves the corollary.
\end{proof}

\begin{proof}[Proof of Corollary~\ref{cor:intermediate-expectation}]
Fix $\delta\in(0,1)$ and $1\le l\le(1-\delta)n$, and set
$X:=s_{n+1-l}(A)$. By Theorem~\ref{thm:main}, after increasing $C$ if
necessary,
\[
  \PP(X>x)
  \le
  \exp(-c\min\{x\sqrt n,n\})
  \qquad
  \text{for }x\ge C\frac l{\sqrt n}.
\]
Hence
\[
\begin{aligned}
  \EE X
  &\le
  C\frac l{\sqrt n}
  +\int_{Cl/\sqrt n}^{C\sqrt n}\PP(X>x)\,dx
  +\EE\bigl[X\mathbf 1_{\{X>C\sqrt n\}}\bigr]  \\
  &\le
  C\frac l{\sqrt n}
  +\int_0^{\sqrt n} e^{-cx\sqrt n}\,dx
  +\EE\bigl[X\mathbf 1_{\{X>C\sqrt n\}}\bigr].
\end{aligned}
\]
The integral is $O(n^{-1/2})$. For the far tail, $X\le\|A\|_{HS}$ and
$\EE\|A\|_{HS}^2=n^2$, while Theorem~\ref{thm:main} gives
$\PP(X>C\sqrt n)\le e^{-cn}$. Thus Cauchy--Schwarz yields
\[
  \EE\bigl[X\mathbf 1_{\{X>C\sqrt n\}}\bigr]
  \le
  (\EE X^2)^{1/2}\PP(X>C\sqrt n)^{1/2}
  \le
  ne^{-cn/2}.
\]
Renaming $c$ gives $ne^{-cn}$.
This is dominated by $Cl/\sqrt n$, and the claim follows.
\end{proof}

\begin{proof}[Proof of Corollary~\ref{cor:rectangular-profile}]
Let $G$ be an independent $N\times(N-n)$ random matrix whose entries are
independent and satisfy the same uniform assumptions with parameters \(a,b\).
Set $\widetilde B=[B\ \ G]\in\mathbb{R}^{N\times N}$.
Since $B^*B$ is a principal submatrix of $\widetilde B^{\,*}\widetilde B$,
Cauchy interlacing gives
\[
  s_{n+1-l}(B)
  \le
  s_{n+1-l}(\widetilde B)
  =
  s_{N+1-(N-n+l)}(\widetilde B).
\]
Applying Theorem~\ref{thm:main} to $\widetilde B$ at level $N-n+l$ gives
\[
  \PP\!\left(s_{n+1-l}(B)>Ct\frac {N-n+l}{\sqrt N}\right)
  \le
  \exp(-c\min\{t(N-n+l),N\}).
\]
For $N-n+l\le(1-\delta)N$,
\[
\begin{aligned}
  \sqrt{N+1}-\sqrt{n-l+1}
  &=
  \frac{N-n+l}{\sqrt{N+1}+\sqrt{n-l+1}}  \\
  &\asymp_\delta
  \frac{N-n+l}{\sqrt N}.
\end{aligned}
\]

For the lower estimate, fix an $N\times(n-l+1)$ column submatrix $B_0$ of
$B$. Interlacing gives $s_{n+1-l}(B)\ge s_{\min}(B_0)$. Let
$\varepsilon_t=\varepsilon_0e^{-t}$, where $\varepsilon_0>0$ is a small constant
depending only on $a$ and $b$. Theorem~\ref{smallest_sing_val_lower} gives
\[
  \PP\!\left(
    s_{\min}(B_0)
    \le
    \varepsilon_0e^{-t}
    \bigl(\sqrt{N+1}-\sqrt{n-l+1}\bigr)
  \right)
  \le
  \left(C\varepsilon_0e^{-t}
       \log\frac{e^t}{\varepsilon_0}\right)^{N-n+l}
  +e^{-cN}.
\]
Choosing $\varepsilon_0$ small enough and decreasing $c$ if necessary gives
\[
  \left(C\varepsilon_0e^{-t}
       \log\frac{e^t}{\varepsilon_0}\right)^{N-n+l}
  \le
  \exp(-ct(N-n+l)),
  \qquad t\ge1.
\]
Also \(e^{-cN}\) is absorbed by
\(\exp(-c\min\{t(N-n+l),N\})\), after decreasing \(c\). Combining the upper and
lower estimates proves
\eqref{eq:rectangular-profile}.
\end{proof}

\bibliographystyle{alpha}
\bibliography{bibliography}

@article {LivTikVer,
    AUTHOR = {Livshyts, Galyna V. and Tikhomirov, Konstantin and Vershynin,
              Roman},
     TITLE = {The smallest singular value of inhomogeneous square random
              matrices},
   JOURNAL = {Ann. Probab.},
  FJOURNAL = {The Annals of Probability},
    VOLUME = {49},
      YEAR = {2021},
    NUMBER = {3},
     PAGES = {1286--1309},
      ISSN = {0091-1798},
   MRCLASS = {60B20 (15B52)},
  MRNUMBER = {4255145},
MRREVIEWER = {Asad Lodhia},
       DOI = {10.1214/20-aop1481},
       URL = {https://doi.org/10.1214/20-aop1481},
}

@article {RV08,
    AUTHOR = {Rudelson, Mark and Vershynin, Roman},
     TITLE = {The {L}ittlewood-{O}fford problem and invertibility of random
              matrices},
   JOURNAL = {Adv. Math.},
  FJOURNAL = {Advances in Mathematics},
    VOLUME = {218},
      YEAR = {2008},
    NUMBER = {2},
     PAGES = {600--633},
      ISSN = {0001-8708},
   MRCLASS = {60E15 (60B20)},
  MRNUMBER = {2407948},
MRREVIEWER = {Ben Joseph Green},
       DOI = {10.1016/j.aim.2008.01.010},
       URL = {https://doi.org/10.1016/j.aim.2008.01.010},
}

@article {GL21,
    AUTHOR = {Livshyts, Galyna V.},
     TITLE = {The smallest singular value of heavy-tailed not necessarily
              i.i.d. random matrices via random rounding},
   JOURNAL = {J. Anal. Math.},
  FJOURNAL = {Journal d'Analyse Math\'{e}matique},
    VOLUME = {145},
      YEAR = {2021},
    NUMBER = {1},
     PAGES = {257--306},
      ISSN = {0021-7670},
   MRCLASS = {60B20},
  MRNUMBER = {4361906},
MRREVIEWER = {Khanh Duy Trinh},
       DOI = {10.1007/s11854-021-0183-2},
       URL = {https://doi.org/10.1007/s11854-021-0183-2},
}

@article{Tatarko18,
  author  = {Tatarko, Kateryna},
  title   = {An upper bound on the smallest singular value of a square random matrix},
  journal = {Journal of Complexity},
  volume  = {48},
  year    = {2018},
  pages   = {119--128},
  doi     = {10.1016/j.jco.2018.06.002}
}

@misc{Dabagia24,
      title={The smallest singular value of inhomogenous random rectangular matrices}, 
      author={Max Dabagia and Manuel Fernandez},
      year={2024},
      eprint={2408.14389},
      archivePrefix={arXiv},
      primaryClass={math.PR},
      url={https://arxiv.org/abs/2408.14389}, 
}

@article{Talagrand96NewLook,
  author  = {Talagrand, Michel},
  title   = {A new look at independence},
  journal = {Annals of Probability},
  volume  = {24},
  number  = {1},
  year    = {1996},
  pages   = {1--34},
  doi     = {10.1214/aop/1042644705}
}

@article{Rogozin61,
  author  = {Rogozin, B. A.},
  title   = {On the increase of dispersion of sums of independent random variables},
  journal = {Theory of Probability \& Its Applications},
  volume  = {6},
  number  = {1},
  year    = {1961},
  pages   = {97--99},
  doi     = {10.1137/1106010}
}

@article {Feng_Wei_intermediate,
    AUTHOR = {Wei, Feng},
     TITLE = {Upper bound for intermediate singular values of random
              matrices},
   JOURNAL = {J. Math. Anal. Appl.},
  FJOURNAL = {Journal of Mathematical Analysis and Applications},
    VOLUME = {445},
      YEAR = {2017},
    NUMBER = {2},
     PAGES = {1530--1547},
      ISSN = {0022-247X},
   MRCLASS = {60B20},
  MRNUMBER = {3545257},
       DOI = {10.1016/j.jmaa.2016.08.007},
       URL = {https://doi.org/10.1016/j.jmaa.2016.08.007},
}

@misc{fernandezDistance2025,
      title={A distance theorem for inhomogenous random rectangular matrices}, 
      author={Manuel Fernandez V},
      year={2025},
      eprint={2408.06309},
      archivePrefix={arXiv},
      primaryClass={math.PR},
      url={https://arxiv.org/abs/2408.06309}, 
}

@book{SteinWeiss,
  author    = {Elias M. Stein and Guido Weiss},
  title     = {Introduction to {F}ourier Analysis on {E}uclidean Spaces},
  series    = {Princeton Mathematical Series},
  number    = {32},
  publisher = {Princeton University Press},
  address   = {Princeton, NJ},
  year      = {1971},
}

@article{Szarek90,
  author  = {Stanis{\l}aw J. Szarek},
  title   = {Spaces with large distance to {$\ell^\infty_n$} and
             random matrices},
  journal = {American Journal of Mathematics},
  volume  = {112},
  number  = {6},
  pages   = {899--942},
  year    = {1990},
  doi     = {10.2307/2374731},
  url     = {https://doi.org/10.2307/2374731}
}

@article{RV09,
  author  = {Mark Rudelson and Roman Vershynin},
  title   = {The smallest singular value of a rectangular random matrix},
  journal = {Communications on Pure and Applied Mathematics},
  volume  = {62},
  number  = {12},
  pages   = {1707--1739},
  year    = {2009},
  doi     = {10.1002/cpa.20294},
  url     = {https://doi.org/10.1002/cpa.20294}
}

@article{RV08_least,
  author  = {Mark Rudelson and Roman Vershynin},
  title   = {The least singular value of a random square matrix
             is {$O(n^{-1/2})$}},
  journal = {Comptes Rendus Math\'{e}matique},
  volume  = {346},
  number  = {15--16},
  pages   = {893--896},
  year    = {2008},
  doi     = {10.1016/j.crma.2008.07.009},
  url     = {https://doi.org/10.1016/j.crma.2008.07.009}
}

@article{TV10_smallest,
  author  = {Tao, Terence and Vu, Van H.},
  title   = {Random matrices: the distribution of the smallest singular values},
  journal = {Geom. Funct. Anal.},
  volume  = {20},
  number  = {1},
  pages   = {260--297},
  year    = {2010},
  doi     = {10.1007/s00039-010-0057-8},
  url     = {https://doi.org/10.1007/s00039-010-0057-8}
}

@article{NguyenVu18,
  author  = {Nguyen, Hoi H. and Vu, Van H.},
  title   = {Normal vector of a random hyperplane},
  journal = {International Mathematics Research Notices},
  volume  = {2018},
  number  = {6},
  pages   = {1754--1778},
  year    = {2018},
  doi     = {10.1093/imrn/rnw273},
  url     = {https://doi.org/10.1093/imrn/rnw273}
}

@article{CacciapuotiMaltsevSchlein13,
  author  = {Cacciapuoti, Claudio and Maltsev, Anna and Schlein, Benjamin},
  title   = {Local {M}archenko--{P}astur law at the hard edge of sample covariance matrices},
  journal = {Journal of Mathematical Physics},
  volume  = {54},
  number  = {4},
  pages   = {043302},
  year    = {2013},
  doi     = {10.1063/1.4801856},
  url     = {https://doi.org/10.1063/1.4801856}
}

@article {AGLPT08,
    AUTHOR = {Adamczak, Rados\l aw and Gu\'{e}don, Olivier and Litvak,
              Alexander and Pajor, Alain and Tomczak-Jaegermann, Nicole},
     TITLE = {Smallest singular value of random matrices with independent
              columns},
   JOURNAL = {C. R. Math. Acad. Sci. Paris},
    VOLUME = {346},
      YEAR = {2008},
    NUMBER = {15-16},
     PAGES = {853--856},
       DOI = {10.1016/j.crma.2008.07.011},
       URL = {https://doi.org/10.1016/j.crma.2008.07.011},
}

@article {BR17,
    AUTHOR = {Basak, Anirban and Rudelson, Mark},
     TITLE = {Invertibility of sparse non-{H}ermitian matrices},
   JOURNAL = {Adv. Math.},
    VOLUME = {310},
      YEAR = {2017},
     PAGES = {426--483},
       DOI = {10.1016/j.aim.2017.02.009},
       URL = {https://doi.org/10.1016/j.aim.2017.02.009},
}

@article {Edelman88,
    AUTHOR = {Edelman, Alan},
     TITLE = {Eigenvalues and condition numbers of random matrices},
   JOURNAL = {SIAM J. Matrix Anal. Appl.},
    VOLUME = {9},
      YEAR = {1988},
    NUMBER = {4},
     PAGES = {543--560},
       DOI = {10.1137/0609055},
       URL = {https://doi.org/10.1137/0609055},
}

@article {Edelman91,
    AUTHOR = {Edelman, Alan},
     TITLE = {The distribution and moments of the smallest eigenvalue of a random matrix of {W}ishart type},
   JOURNAL = {Linear Algebra Appl.},
    VOLUME = {159},
      YEAR = {1991},
     PAGES = {55--80},
       DOI = {10.1016/0024-3795(91)90076-9},
       URL = {https://doi.org/10.1016/0024-3795(91)90076-9},
}

@article{GLPTJ17,
    AUTHOR = {Gu\'{e}don, Olivier and Litvak, Alexander E. and Pajor, Alain and Tomczak-Jaegermann, Nicole},
     TITLE = {On the interval of fluctuation of the singular values of random matrices},
   JOURNAL = {J. Eur. Math. Soc. (JEMS)},
    VOLUME = {19},
      YEAR = {2017},
    NUMBER = {5},
     PAGES = {1469--1505},
       DOI = {10.4171/JEMS/697},
       URL = {https://doi.org/10.4171/JEMS/697},
}

@article {LR12,
    AUTHOR = {Litvak, Alexander E. and Rivasplata, Omar},
     TITLE = {Smallest singular value of sparse random matrices},
   JOURNAL = {Studia Math.},
    VOLUME = {212},
      YEAR = {2012},
    NUMBER = {3},
     PAGES = {195--218},
       DOI = {10.4064/sm212-3-1},
       URL = {https://doi.org/10.4064/sm212-3-1},
}

@article{LPRT05,
    AUTHOR = {Litvak, A. E. and Pajor, A. and Rudelson, M. and Tomczak-Jaegermann, N.},
     TITLE = {Smallest singular value of random matrices and geometry of random polytopes},
   JOURNAL = {Adv. Math.},
    VOLUME = {195},
      YEAR = {2005},
    NUMBER = {2},
     PAGES = {491--523},
       DOI = {10.1016/j.aim.2004.08.004},
       URL = {https://doi.org/10.1016/j.aim.2004.08.004},
}

@article {RT18,
    AUTHOR = {Rebrova, Elizaveta and Tikhomirov, Konstantin},
     TITLE = {Coverings of random ellipsoids, and invertibility of matrices
              with i.i.d. heavy-tailed entries},
   JOURNAL = {Israel J. Math.},
    VOLUME = {227},
      YEAR = {2018},
    NUMBER = {2},
     PAGES = {507--544},
       DOI = {10.1007/s11856-018-1732-y},
       URL = {https://doi.org/10.1007/s11856-018-1732-y},
}

@article {Szarek91,
    AUTHOR = {Szarek, Stanis\l aw J.},
     TITLE = {Condition numbers of random matrices},
   JOURNAL = {J. Complexity},
    VOLUME = {7},
      YEAR = {1991},
    NUMBER = {2},
     PAGES = {131--149},
       DOI = {10.1016/0885-064X(91)90002-F},
       URL = {https://doi.org/10.1016/0885-064X(91)90002-F},
}

@article {TV09,
    AUTHOR = {Tao, Terence and Vu, Van H.},
     TITLE = {Inverse {L}ittlewood-{O}fford theorems and the condition
              number of random discrete matrices},
   JOURNAL = {Ann. of Math. (2)},
    VOLUME = {169},
      YEAR = {2009},
    NUMBER = {2},
     PAGES = {595--632},
       DOI = {10.4007/annals.2009.169.595},
       URL = {https://doi.org/10.4007/annals.2009.169.595},
}

@article{RebrovaVershynin2018,
  author  = {Rebrova, Elizaveta and Vershynin, Roman},
  title   = {Norms of random matrices: local and global problems},
  journal = {Advances in Mathematics},
  volume  = {324},
  pages   = {40--83},
  year    = {2018},
  doi     = {10.1016/j.aim.2017.11.001},
  url     = {https://doi.org/10.1016/j.aim.2017.11.001}
}

@article{JainSahSawhney2022,
  author  = {Jain, Vishesh and Sah, Ashwin and Sawhney, Mehtaab},
  title   = {Optimal and algorithmic norm regularization of random matrices},
  journal = {Proceedings of the American Mathematical Society},
  volume  = {150},
  number  = {10},
  pages   = {4503--4518},
  year    = {2022},
  doi     = {10.1090/PROC/15964},
  url     = {https://doi.org/10.1090/PROC/15964}
}

@article{BEKYY14,
  author  = {Bloemendal, Alex and Erd\H{o}s, L\'aszl\'o and Knowles, Antti and Yau, Horng-Tzer and Yin, Jun},
  title   = {Isotropic local laws for sample covariance and generalized {W}igner matrices},
  journal = {Electronic Journal of Probability},
  volume  = {19},
  number  = {33},
  pages   = {1--53},
  year    = {2014},
  doi     = {10.1214/EJP.v19-3054},
  url     = {https://doi.org/10.1214/EJP.v19-3054}
}

@article{PY14,
  author  = {Pillai, Natesh S. and Yin, Jun},
  title   = {Universality of covariance matrices},
  journal = {The Annals of Applied Probability},
  volume  = {24},
  number  = {3},
  pages   = {935--1001},
  year    = {2014},
  doi     = {10.1214/13-AAP939},
  url     = {https://doi.org/10.1214/13-AAP939}
}

@article{AEK14,
  author  = {Ajanki, Oskari H. and Erd\H{o}s, L\'aszl\'o and Kr\"uger, Torben},
  title   = {Local semicircle law with imprimitive variance matrix},
  journal = {Electronic Communications in Probability},
  volume  = {19},
  year    = {2014},
  doi     = {10.1214/ECP.v19-3121},
  url     = {https://doi.org/10.1214/ECP.v19-3121}
}

@article{KM23,
  author  = {Kafetzopoulos, Anastasis and Maltsev, Anna},
  title   = {Local {M}archenko--{P}astur law at the hard edge of the sample covariance ensemble},
  journal = {Journal of Mathematical Physics},
  volume  = {64},
  number  = {12},
  pages   = {123501},
  year    = {2023},
  doi     = {10.1063/5.0121895},
  url     = {https://doi.org/10.1063/5.0121895}
}
\end{document}